\RequirePackage{fix-cm}

\documentclass[smallextended]{svjour3}       % onecolumn (second format)
\smartqed  % flush right qed marks, e.g. at end of proof
\usepackage{graphicx,amssymb}
\usepackage{amsmath,xcolor,color,soul}
\newcommand{\R}{\mathbb{R}}

\newcommand{\vp}{\varphi}

\newcommand{\e}{\varepsilon}

\newcommand{\ratioo}{\mathcal{G}_0}

\newcommand{\ep}{\varepsilon}

\begin{document}

\title{A Poincar\'e type inequality with three constraints%\thanks{Grants or other notes
}


\author{
Gisella Croce         
\and
Antoine Henrot
}


\institute{G. Croce \at
              Normandie Univ, France; ULH, LMAH, F-76600 Le Havre; FR CNRS 3335, 25 rue
Philippe Lebon, 76600 Le Havre, France
\email{gisella.croce@univ-lehavre.fr}                  
           \and
     A. Henrot \at
Institut \'Elie Cartan de Lorraine UMR CNRS 7502,
Universit\'e de Lorraine, BP 70239 54506 Vandoeuvre-les-Nancy Cedex, France
\email{antoine.henrot@univ-lorraine.fr}
}
\date{Received: date / Accepted: date}

\maketitle

\begin{abstract}
In this paper, we consider a problem in calculus of variations motivated
by a quantitative isoperimetric inequality in the plane.
More precisely, the aim of this article is the computation of the minimum
of
 the variational problem 
$$
\inf_{u\in\mathcal{W}}
\frac{\displaystyle\int_{-\pi}^{\pi}[(u')^2-u^2]d\theta}{\displaystyle\left[\int_{-\pi}^{\pi}
|u| d\theta\right]^2}
$$
where a function $u\in \mathcal{W}$ is a  $H^1(-\pi,\pi)$ periodic function, with
 zero average on $(-\pi,\pi)$
 and orthogonal to sine and cosine.
\end{abstract}

\section{Introduction}
In this article we are interested in the following variational problem :
$$
\inf_{u\in \mathcal{W}}
\frac{\displaystyle\int_{-\pi}^{\pi}[(u')^2-u^2]d\theta}{\displaystyle\left[\int_{-\pi}^{\pi}
|u| d\theta\right]^2}
$$
where $ \mathcal{W}$ denotes the subspace of functions in the Sobolev space
$H^1(-\pi,\pi)$ 
that are $2\pi$-periodic, satisfying
 the following constraints:
\begin{itemize}
\item[(L1)]
$\displaystyle \int_{-\pi}^{\pi} u(\theta)\,d\theta=0$
\item[(L2)]
$\displaystyle\int_{-\pi}^{\pi} u(\theta) \cos(\theta)\,d\theta=0$
\item[(L3)]
$\displaystyle \int_{-\pi}^{\pi}u(\theta)\sin(\theta)\,d\theta=0$.
\end{itemize}
Our aim is to compute the value of the minimum and to identify the minimizer.
The difficulty comes here from the nonlinear term in the denominator of the functional together
with the three constraints (L1), (L2), (L3).
We will prove  the following result. Let
\begin{equation}\label{opepl}
m=\inf_{u\in\mathcal{W}}J(v)\,,\quad
J(v)=\frac{\displaystyle \int_{-\pi}^\pi (|v'|^2-|v|^2)}{\displaystyle \left[\int_{-\pi}^\pi |v|\right]^2}
\,.
\end{equation}

\begin{theorem}\label{mainthm}
Let $m$ be defined by (\ref{opepl}).
Then $\displaystyle m=\frac{1}{2(4-\pi)}$ and the minimizer $u$ of the functional $J$ is the odd
and $\pi$ periodic
function defined on $[0,\pi/2]$ by $u(\theta)=\cos \theta+ \sin \theta -1$.
\end{theorem}
%A direct consequence of our theorem is that, by (\ref{conclusion}),
%$\displaystyle \liminf_{\e \to 0} m_\e \geq 0.457>0.406$.

We remark that the minimization problem (\ref{opepl}) is a variant of the
Wirtinger inequality :
$$
\inf_{u \in H^1_{per}(-\pi,\pi): \int_{-\pi}^\pi u=0}\frac{ \displaystyle
\int_{-\pi}^{\pi}|u'|^2d\theta}{\displaystyle
\int_{-\pi}^{\pi} |u|^2 d\theta} = 1\,.
$$
%because of the constraints (L2) and (L3) and the $L^1(-\pi,\pi)$ norm at
%the denominator. 
At first glance, one could think that $\cos 2\theta$ is a minimizer
of the functional in \eqref{opepl}, as for the  Wirtinger-type inequality
$$
\inf_{u\in\mathcal{W}}\frac{ \displaystyle \int_{-\pi}^{\pi}[(u')^2-u^2]d\theta}{\displaystyle \int_{-\pi}^{\pi} u^2 d\theta}\,,
$$
but this is not true.

 In the literature one can find various generalizations of the Wirtinger inequality,
without our constraints (L2) and (L3).
In the series of papers
 \cite{belloni-kawohl}, \cite{busl}, \cite{croce-dacorogna}, \cite{DacGanSub},
\cite{egorov}, \cite{gerasimov-nazarov},
 \cite{ghisi-rovellini}, \cite{kawohl},
 \cite{mukoseeva-nazarov}, \cite{Nazarov2002},
the authors  consider different norms of $u'$ and $u$ and on the mean value
of $u$, namely
$$
\inf_{u \in W^{1,p}_{per}(-\pi,\pi): \int_{-\pi}^\pi |u|^{r-2}u=0}\frac{
\displaystyle \int_{-\pi}^{\pi}|u'|^pd\theta}{\displaystyle
\int_{-\pi}^{\pi} |u|^q d\theta}
$$
for all values of $p, q, r$ greater than 1.
We also mention 
\cite{ferone-nitsch-trombetti} in
which the authors study, in any dimension $N\geq 1$, the inequality
$$
\inf_{u \in W^{1,p}(\Omega): \int_{\Omega}|u|^{p-2}u\omega=0}
\frac{\displaystyle \int_{\Omega}|\nabla u|^p\omega}{\displaystyle \int_{\Omega}|u|^p\omega}
$$
with a positive log-concave 
weight $\omega$, on a convex bounded domain $\Omega\subset \mathbb{R}^N$.
\\
In our article, the Rayleigh quotient that we minimise  is not "too nonlinear"
as in these cited papers. The difficulty comes  from  the orthogonality to
sine and cosine.

We will explain the strategy of the proof in the next section. The proof will be developed in Sections 3, 4, 5 and 6.
In the last section we will explain the geometrical motivation of this minimization problem and how we used the value of $m$ to study a quantitative isoperimetric inequality. 

%%%%%%%%%%%%%%%%%%
\section{Strategy of the proof of Theorem \ref{mainthm}}
The strategy to prove our result is the following. It is immediate to see
that the minimization problem (\ref{opepl}) has a solution. Thus
we write the Euler equation that any minimizer $u$ satisfies:
$$
-u''-u=m \cdot sgn(u) +\lambda_0 +\lambda_1 \cos \theta+\lambda_2 \sin \theta\,.
$$ 
For that purpose,
we introduce the three Lagrange multipliers, related to the three constraints
(L1), (L2) and (L3) : they can be written as a function of $sgn(u)$. 
$$
\lambda_0=-\frac{m}{2\pi} \int_{-\pi}^{\pi} sgn(u(\theta))d\theta
$$
$$
\lambda_1=-\frac{m}{\pi} \int_{-\pi}^{\pi} sgn(u(\theta)) \cos
\theta d\theta 
$$
$$
\lambda_2=-\frac{m}{\pi} \int_{-\pi}^{\pi} sgn(u(\theta)) \sin
\theta d\theta.
$$
By homogeneity, we can assume that the $L^1(-\pi,\pi)$ norm of $u$ equals 1 and, by a translation, that $\lambda_2$ is zero.
Our aim is to prove that 
$$
\lambda_0=\lambda_1=0
$$ 
and that $u$ has four nodal intervals, of same length.
This allows us to fully determine the minimizer and compute the value of $m$.
Indeed, by using the explicit expression of $u$ on any of the four nodal domain, as
a function of the endpoints of the interval, we can easily deduce the explicit expression of the solution $u$ on the whole interval $[-\pi,\pi]$ and thus compute the value of $m$ (see Section 6).
\\
For that purpose, the most involved step is to prove that $\lambda_1$ is zero (see Proposition \ref{proplambda_1}). We are now going to give an idea of the strategy.
\\
Let $I_k=[a_{k-1},a_k]$, $I_{k+1}=[a_{k},a_{k+1}]$, $I_{k+2}=[a_{k+1},a_{k+2}]$,
$I_{k+3}=[a_{k+2},a_{k+3}]$ 
four consecutive intervals of lengths 
$\ell_{k},\ell_{k+1},\ell_{k+2},\ell_{k+3}$, respectively. We assume that
$u$ is alternatively
 positive, negative, positive and negative on these intervals.
From the explicit expression of $u$ on any nodal domain, as
a function of the endpoints of the interval, it is easy to deduce an equality involving the length of an interval and its consecutive :  
$$
\lambda_0 \sin \frac{\ell_{k}+\ell_{k+1}}{2} - m \sin \frac{\ell_{k+1}-\ell_{k}}{2}
=
\frac{\lambda_1}{2}
\cos \frac{\ell_{k}}{2} \cos \frac{\ell_{k+1}}{2} A(I_{k},I_{k+1})\,.
$$
$$
\lambda_0 \sin \frac{\ell_{k+1}+\ell_{k+2}}{2} + m \sin \frac{\ell_{k+2}-\ell_{k+1}}{2}
=
\frac{\lambda_1}{2}
\cos \frac{\ell_{k+1}}{2} \cos \frac{\ell_{k+2}}{2} A(I_{k+1},I_{k+2})
$$
$$
\lambda_0 \sin \frac{\ell_{k+2}+\ell_{k+3}}{2} - m \sin \frac{\ell_{k+3}-\ell_{k+2}}{2}
=
\frac{\lambda_1}{2}
\cos \frac{\ell_{k+2}}{2} \cos \frac{\ell_{k+3}}{2} A(I_{k+2},I_{k+3})\,,
$$
where
$$A(I_k,I_{k+1})=\frac{\ell_k}{\sin\ell_k} \sin a_{k-1} - \frac{\ell_{k+1}}{\sin\ell_{k+1}}
\sin a_{k+1}
$$
(see Section \ref{subgenral}).
Assuming by contradiction that $\lambda_1\neq 0$ and that $\ell_k\neq \ell_{k+2}, \ell_{k+1}\neq \ell_{k+3}$,
this $3\times 3$ system in $(\lambda_0, \lambda_1,m)$ has necessarily a null determinant, that provides, after some manipulations the identity
$$
\frac{C}{\sin \frac{\ell_{k+3}}{2} \cos \frac{\ell_{k}}{2}} \left[\cos \frac{\ell_{k+1}}{2}
\sin \frac{\ell_{k+3}}{2} \sin \frac{\ell_{k}+\ell_{k+2}}{2} + \cos \frac{\ell_{k}}{2}
\sin \frac{\ell_{k+2}}{2} \sin \frac{\ell_{k+1}+\ell_{k+3}}{2}\right]
$$
$$
=
A(I_{k},I_{k+1}) \cos \frac{\ell_{k+1}}{2} \cos \frac{\ell_{k+2}}{2}
- A(I_{k+2},I_{k+3}) \frac{\cos \frac{\ell_{k+2}}{2} \cos \frac{\ell_{k+3}}{2}\sin
\frac{\ell_{k+1}}{2}}{\sin
\frac{\ell_{k+3}}{2}}\,,
$$
where $\displaystyle C=\frac{2(m+\lambda_0)}{\lambda_1}$.
After proving that
the length of each nodal domain is less than $\pi$, we will be able to study the sign of each term and arrive to 
the contradiction that  $C$  is both positive and negative.
\\
The proof that the length of each nodal domain is strictly less than $\pi$ is more difficult that we expected and is done in Section \ref{lengthpi}.

%%%%%%%%%%%%%%%%%%%%%%%%%%%%%%%%%%
%%%%%%%%%%%%%%%%%%%%%%%%%%%%%%%%
%%%%%%%%%%%%%%%%%%%%%%%%%%%%%%%%%%%%%%
\section{Preliminaries}\label{preliminaries}
\subsection{Existence of a minimizer and Euler equation}
The existence of a minimizer of  the functional in (\ref{opepl}) and the optimality conditions follow
easily from the direct methods of the calculus of variations.

Notice that we will assume that 
\begin{equation}\label{normaL1=1}
\displaystyle \int_{-\pi}^{\pi}|u|=1
\end{equation}
 in all the paper.

\begin{proposition}\label{prop_multiplicateurs_Lagrange}
The minimization problem \eqref{opepl} has a solution $u$. If  $K$ denotes the set of points where $u$
vanishes, then $K$ has zero Lebesgue measure and $u$ satisfies the following Euler 
equation almost everywhere in $(-\pi,\pi)$:
\begin{equation}\label{eulereq}
-u''-u=m \cdot sgn(u) +\lambda_0 +\lambda_1 \cos \theta+\lambda_2 \sin \theta\,,
\end{equation}
where the Lagrange multipliers are given by
\begin{equation}\label{lagrange}
\begin{array}{c} \vspace{2mm}
\displaystyle\lambda_0=-\frac{m}{2\pi} \int_{-\pi}^{\pi} sgn(u(\theta))d\theta
\\ \vspace{2mm}
\displaystyle\lambda_1=-\frac{m}{\pi} \int_{-\pi}^{\pi} sgn(u(\theta)) \cos
\theta d\theta \\ \vspace{2mm}
\displaystyle\lambda_2=-\frac{m}{\pi} \int_{-\pi}^{\pi} sgn(u(\theta)) \sin
\theta d\theta.
\end{array}
\end{equation}
In particular, the function $u$ is $C^1(-\pi,\pi)$.
\end{proposition}
\begin{proof}
The existence of a solution to problem \eqref{opepl} is straightforward using the classical methods of the calculus of variations. We are going to write the optimality condition. For this purpose, 
we introduce the open set $\omega=K^c=\{x\in (-\pi,\pi), u(x)\not=0\}$. We fix now a function $\vp \in H^1((-\pi,\pi))$ satisfying
$$\int_{-\pi}^\pi \vp(x) dx = \int_{-\pi}^\pi \vp(x) \cos x dx = \int_{-\pi}^\pi \vp(x) \sin x dx = 0\,.$$
Therefore $u+t\varphi \in \mathcal{W}$, that is, it can be used as a test function for our functional 
$J(v)=\displaystyle \frac{\int_{-\pi}^\pi (|v'|^2-|v|^2)}{\left[\int_{-\pi}^\pi |v|\right]^2}$.
We observe that 
$$\displaystyle
\int_{-\pi}^\pi |u+t\varphi| =\int_\omega |u+t\varphi| + |t|\int_K |\varphi|.$$
Now, on the set $\omega$ where $u$ is not zero, we have the expansion
$$\int_\omega |u+t\\vp| = \int_\omega |u| + t\int_\omega sign(u) \vp +o(t).$$
Therefore, we get
\begin{equation}\label{euler1}
J(u+t\vp)=J(u) +2t\left[\int_{-\pi}^{\pi} (u'\vp'-u\vp)- \int_\omega m sign(u) \vp  \right]-2|t|m\int_K |\vp| +o(t)\,.
\end{equation}
Let us denote by $I_0$ the term $\displaystyle \int_{-\pi}^{\pi} (u'\vp'-u\vp)- \int_\omega m sign(u) \vp$. 
If $I_0>0$, we choose $t<0$ small enough  and get a contradiction.
If $I_0<0$, we choose $t>0$ small enough  and get the same contradiction. Therefore $I_0=0$ for all admissible $\vp$ providing on $\omega$
the desired Euler equation. At last, coming back to \eqref{euler1} we necessarily get $\int_K |\vp|=0$ proving that $K$ has zero measure
and the Euler equation holds {almost everywhere}.

The expression of the Lagrange multipliers is obtained by integrating the Euler equation after
multiplication by $1,\cos\theta,\sin\theta$. 

The $C^1$ regularity of the function $u$ comes from the Euler
equation that shows that its second derivative is $L^\infty$, implying that
$u\in W^{2,\infty}(-\pi,\pi) \subset C^1(-\pi,\pi)$.
\end{proof}

\begin{remark}
Using \eqref{lagrange} we see that
\begin{equation}\label{m+l0}
m+\lambda_0 =\frac{m}{2\pi} \int_{-\pi}^{\pi} (1-sgn(u(\theta))) d\theta
>0
\end{equation}
and
\begin{equation}\label{m-l0}
-m+\lambda_0 =-\frac{m}{2\pi} \int_{-\pi}^{\pi} (1+sgn(u(\theta))) d\theta
<0\,.
\end{equation}
\end{remark}
Up to a translation on $\theta$, we can assume  that one Lagrange multiplier
is zero. Indeed, by periodicity,
replacing $u(\theta)$ by $u(\theta+a)$ amounts to replace $\lambda_2$
by $\cos a\lambda_2 - \sin a \lambda_1$. Thus we can  choose $a$ such that
$\lambda_2=0$.
Therefore in the sequel, we will assume:
\begin{equation}\label{lam2}
\mbox{the Lagrange multiplier } \lambda_2 \mbox{ is zero}.
\end{equation}
We introduce the measure of the sets where $u$ is respectively positive and negative:
\begin{equation}\label{defl+}
\ell_+=|\{x\in (-\pi,\pi): u(x)\geq 0\}|,\ \;\ell_-=|\{x\in (-\pi,\pi): u(x)<0\}|\,.
\end{equation}
With these notations, we can rewrite $m+\lambda_0$ and $m-\lambda_0$:
\begin{equation}\label{ml0}
m+\lambda_0=\frac{m\ell_-}{\pi},\quad m-\lambda_0=\frac{m\ell_+}{\pi}=\frac{m}{\pi}(2\pi-\ell_-).
\end{equation}
Note also that if $u(x)$ is a minimizer of our problem, then $-u(x)$ or $u(-x)$
or $-u(-x)$
are also minimizers. Therefore, without loss of generality, we can assume,
from now on, that $$\ell_+\geq\ell_-\,.$$

\subsection{Expression of the solution}\label{subgenral}
As usual, we call nodal domain, each interval on which $u$ has a constant
sign. We observe that $u$ can be zero in some points in the interior of a nodal domain. 
%However $u$ cannot be identically 0 in a whole interval, say $(\alpha, \beta)$. Indeed, let $\varphi \in C^1_0(-\pi,\pi)$, such that $supp (\varphi)=(\alpha,\beta)$ and 
%$$
%\displaystyle \int_{\alpha}^{\beta} \varphi = \int_{\alpha}^{\beta} \varphi(x)\cos x= \int_{\alpha}^{\beta} \varphi(x)\sin x=0\,.
%$$ 
%Therefore $u+t\varphi \in \mathcal{W}$, that is, it can be used as a test function for our functional 
%$J(v)=\displaystyle \frac{\int_{-\pi}^\pi (|v'|^2-|v|^2)}{\left[\int_{-\pi}^\pi |v|\right]^2}$.
%We observe that 
%$$\displaystyle
%\int_{-\pi}^\pi |u+t\varphi| =|t|\int_{\alpha}^\beta |\varphi|+ \int_{-\pi}^\pi |u|.$$
%Moreover 
%$$\displaystyle \int_{-\pi}^\pi |u'+t\varphi'|^2- |u+t\varphi|^2 
%$$
%$$
%=\int_{-\pi}^\pi |u'|^2+t^2\int_{-\pi}^\pi |\varphi'|^2 +2t \int_{-\pi}^\pi u'\varphi' -
%\int_{-\pi}^\pi |u|^2 -t^2\int_{-\pi}^\pi |\varphi|^2 -2t\int_{-\pi}^\pi u\varphi
%$$
%$$=
%\int_{-\pi}^\pi |u'|^2+t^2\int_{-\pi}^\pi |\varphi'|^2 -
%\int_{-\pi}^\pi |u|^2 -t^2\int_{-\pi}^\pi |\varphi|^2 
%.$$
%This implies that
%$$
%\displaystyle J(u)\leq J(u+t\varphi) =\displaystyle \frac{\displaystyle\int_{-\pi}^\pi (|u'|^2-|u|^2 )+O(t^2)}{\displaystyle\left[\int_{-\pi}^\pi |u|\right]^2 +2|t| \int_\alpha^\beta |\varphi| +O(t^2)}< J(u)\,,\quad t<<1
%$$
%where  $O(t^2)$ is smaller than $Ct^2$ for some positive constant $C$. This is a contradiction.}

By periodicity, there is an even number of nodal domains. 
A straight consequence of the Sturm-Hurwitz theorem (see \cite{Katriel} and \cite{hurwitz}) applied to any minimizer (satisfying (L1), (L2), (L3))
is that
\begin{proposition}\label{4dom}
A minimizer $u$ has at least four nodal domains.
\end{proposition}

 The main difficulty will
be to prove that there are {\it exactly} four nodal domains with same length.
It will be a consequence of the fact that the Lagrange multipliers are all zero and will
be done in Section \ref{Conclusion}.

{On each nodal domain}, we can integrate the Euler equation
and get an explicit expression of the solution. 
%\textcolor{red}{Cette phrase n'est pas correcte pour moi. En effet, on peut int\'egrer l'\'equation d'Euler sur un intervalle o\`u $u$ est strictement positive ou strictement n\'egative, pour pouvoir remplacer $sign(u)$ par 1 ou -1....je propose une nouvelle version de cette partie,  en bleu.}
We are going to write explicitly $u$ on two consecutive intervals $[a,b],
[b,c]$, where 
$$
u(a)=u(b)=u(c)=0
$$
and
$$
u\geq 0\, \textnormal{in}\, [a,b],\quad u\leq 0\, \textnormal{in}\, [b,c]\,.
$$

Assume that $u\geq 0$  on $(a,b)$, $u(a)=u(b)=0$.
By integrating the Euler equation on $[a,b]$ and using that $u(a)=u(b)=0$,
we find
$$
u(x)=A_0 \cos x+B_0 \sin x -(m+\lambda_0) - \frac{\lambda_1}{2} x \sin
x, \,\,\,\,\,x\in [a,b]
$$
where 
\begin{equation}\label{expressionsA0B0}
\begin{array}{c}
\displaystyle
A_0=(m+\lambda_0) \frac{\cos (\frac{a+b}{2})}{\cos (\frac{b-a}{2})} -\frac{\lambda_1}{2}\frac{(b-a)\sin
a \sin b}{\sin(b-a)}\,,
\\
\displaystyle
B_0=(m+\lambda_0) \frac{\sin (\frac{a+b}{2})}{\cos (\frac{b-a}{2})} +\frac{\lambda_1}{2}\frac{b\sin b\cos a - a\sin a \cos b}{\sin(b-a)}\,.
\end{array}
\end{equation}
%Now, if $u$ is strictly positive on a consecutive interval, say $(\beta, \gamma)$ and $u(\gamma)=0$, the above expression of $u$ is still valid on the entire interval $(\alpha,\gamma)$. In the expression of the constants $A$ and $B$ one has only to replace $\beta$ with $\gamma$. 
%Therefore we can assume that 
%\begin{equation}\label{expression_u_positive}
%u(x)=A_0 \cos x+B_0 \sin x -(m+\lambda_0) - \frac{\lambda_1}{2} x \sin
%x, \,\,\,\,\,x\in [a,b]
%\end{equation}
%where 
%\begin{equation}\label{expressionsA0B0}
%\begin{array}{c}
%\displaystyle
%A_0=(m+\lambda_0) \frac{\cos (\frac{a+b}{2})}{\cos (\frac{b-a}{2})} -\frac{\lambda_1}{2}\frac{(b-a)\sin
%a \sin b}{\sin(b-a)}\,,
%\\
%\displaystyle
%B_0=(m+\lambda_0) \frac{\sin (\frac{a+b}{2})}{\cos (\frac{b-a}{2})} +\frac{\lambda_1}{2}\frac{b
%\sin b \cos a - a \sin a \cos b}{\sin(b-a)}
%\end{array}
%\end{equation}
%is the expression of the solution on a nodal domain where $u\geq 0$ and $u(a)=u(b)=0$.
%}
Assume that $u\leq 0$ on an interval $[b,c]$, with $u(b)=u(c)=0$. We find
$$
u(x)=A_1 \cos x+B_1 \sin x -(-m+\lambda_0) - \frac{\lambda_1}{2} x \sin
x, \,\,\,\,x\in [b,c]$$
where 
\begin{equation}\label{expression1A1B1}
\begin{array}{c}
\displaystyle
A_1=(-m+\lambda_0) \frac{\cos (\frac{c+b}{2})}{\cos (\frac{c-b}{2})} -\frac{\lambda_1}{2}\frac{(c-b)\sin
c \sin b}{\sin(c-b)}\,,
\\
\displaystyle
B_1=(-m+\lambda_0) \frac{\sin (\frac{b+c}{2})}{\cos (\frac{c-b}{2})} +\frac{\lambda_1}{2}\frac{c
\sin c \cos b - b \sin b \cos c}{\sin(c-b)}\,.
\end{array}
\end{equation}
Now we can obtain another expression of the solution on $[b,c]$ using
 the $C^1$ regularity of $u$ in
 $b$. This gives
$$
A_1 \cos b +B_1 \sin b = \lambda_0 - m + \frac{\lambda_1}{2} b\sin b
$$
and 
$$
(A_0-A_1) \sin b=(B_0 - B_1)\cos b\,.
$$
We then get a different expression for
$A_1$ and $B_1$:
$$
A_1=(\lambda_0 - m)\cos b+\frac{\lambda_1}{2} b \sin b \cos b -B_0 \sin b
\cos b +A_0 \sin^2 b
$$
$$
B_1=(\lambda_0 - m)\sin b+\frac{\lambda_1}{2} b \sin^2 b +B_0 \cos^2 b -A_0
\sin b \cos b\,.
$$
Replacing $A_0, B_0$ of formulas (\ref{expressionsA0B0})
in the above expressions of $A_1$ and $B_1$, one gets:
\begin{equation}\label{expression2A1B1}
\begin{array}{c}
\displaystyle
A_1 = \lambda_0 \frac{\cos(\frac{a+b}{2})}{\cos(\frac{b-a}{2})}-m \frac{\cos(\frac{3b-a}{2})}{\cos(\frac{b-a}{2})}-\frac{\lambda_1}{2}(b-a)\frac{\sin
a \sin b}{\sin(b-a)}
\\
\displaystyle
B_1 = \lambda_0 \frac{\sin(\frac{a+b}{2})}{\cos(\frac{b-a}{2})}-m \frac{\sin(\frac{3b-a}{2})}{\cos(\frac{b-a}{2})}+\frac{\lambda_1}{2}\frac{b\sin
b\cos a-a\sin a \cos b}{\sin(b-a)}\,.
\end{array}
\end{equation}
Expressions 
(\ref{expression1A1B1}) and (\ref{expression2A1B1}) of $A_1$ give
$$
\lambda_0\frac{\sin b\sin(\frac{c-a}{2})}{\cos(\frac{b-a}{2})\cos(\frac{c-b}{2})}
-m
\frac{\sin b \sin(\frac{a+c-2b}{2})}{\cos(\frac{b-a}{2})\cos(\frac{c-b}{2})}=\frac{\lambda_1}{2}\sin
b
\left[
\frac{(b-a)\sin a}{\sin(b-a)}-\frac{(c-b)\sin c}{\sin(c-b)}
\right]\,.
$$
Let us assume now that $b\neq 0$. If $\ell_1=b-a$ and $\ell_2=c-b$, this
equality can be written as
\begin{equation}\label{derniere_relation}
\lambda_0 \sin \frac{\ell_1+\ell_2}{2} - m \sin \frac{\ell_2-\ell_1}{2} =
\frac{\lambda_1}{2}
\cos \frac{\ell_1}{2} \cos \frac{\ell_2}{2} \left[\frac{\ell_1}{\sin \ell_1}\sin
a -
\frac{\ell_2}{\sin \ell_2}\sin c\right]\,.
\end{equation}
Let us assume now that $b=0$.
Expressions 
(\ref{expression1A1B1}) and (\ref{expression2A1B1}) of $B_1$ give
$$
(-m+\lambda_0)\tan\left(\frac c2\right)+
\frac{\lambda_1}{2}c=
(\lambda_0+m)\tan\left(\frac a2\right)+
\frac{\lambda_1}{2}a
$$
that is,
$$
\lambda_0 \sin\left(\frac{c-a}{2}\right) - m\sin\left(\frac{a+c}{2}\right)+\frac{\lambda_1}{2}(c-a)\cos\left(\frac{c}{2}\right)\cos\left(\frac{a}{2}\right)=0
$$
This is exactly equation (\ref{derniere_relation}) written in the case $b=0$.

Here we have assumed the lengths of the intervals
not equal to $\pi$. The case of an interval of length $\pi$ will be considered
in Section
\ref{lengthpi}.

%%%%%%%%%%%%%%%%%%%%%%%%%%%
%%%%%%%%%%%%%%%%%%%%%%%%%%%%%%%%%%%%%
\section{The length of the nodal intervals cannot be greater than $\pi$}\label{lengthpi}
In this section we prove that the length of any nodal interval
of the solution $u$ is strictly less than $\pi$.  We argue by contradiction,  mainly by considering the integral of $u$ on a nodal domain.
\\
We assume that there exists a nodal interval $(a,b)$
of length $\ell$ greater than $\pi$. Without loss of generality, we can assume
that 
\begin{itemize}
\item
$u\geq 0$ on $(a,b)$;
\item
 $a\in [-\pi,0]$ and $b\in (0,\pi]$ (since the function $x\mapsto u(x+\pi)$
is also
a minimizer satisfying $\lambda_2=0$); 
\item
 $\displaystyle \frac{a+b}{2}\leq 0$
(since $u(-x)$ is also a minimizer); this implies $\displaystyle a\leq -\frac{\pi}{2}$.
\end{itemize}
In the sequel we will
 call {\it negative interval} (resp. {\it positive interval}) any interval
where $u$
is negative (resp. positive). 
\\
On a negative interval $(a_j,b_j)$ of length $\ell_j\neq\pi$, we have
\begin{equation}\label{gr1n}
\int_{a_j}^{b_j} u(x) dx= (-m+\lambda_0)\left(2\tan \frac{\ell_j}{2} -\ell_j\right)+\frac{\lambda_1}{2}
\left(2\sin \frac{\ell_j}{2} \cos \frac{a_j+b_j}{2}\right)\left[1+\frac{\ell_j}{\sin
\ell_j}\right]\,,
\end{equation}
while, on a positive interval $(a_k,b_k)$  of length $\ell_k\neq \pi$, we
have
\begin{equation}\label{gr1p}
\int_{a_k}^{b_k} u(x) dx= (m+\lambda_0)\left(2\tan \frac{\ell_k}{2} -\ell_k\right)
+\frac{\lambda_1}{2}
\left(2\sin \frac{\ell_k}{2} \cos \frac{a_k+b_k}{2}\right)\left[1+\frac{\ell_k}{\sin
\ell_k}\right]\,.
\end{equation}
In the  case of a nodal domain of length $\pi$, let 
$(a,a+\pi)$
be such an interval where we suppose $u\geq 0$. Now the Euler equation
$$
\left\{
\begin{array}{l}
-u''-u=m +\lambda_0 +\lambda_1 \cos x \ \mbox{on } (a,a+\pi) \\
u(a)=0,\ u(a+\pi)=0
\end{array}
\right.
$$
has not a unique solution, since 1 is an eigenvalue on the interval. 
Moreover, by the Fredholm alternative, the right-hand side of the equation
must be orthogonal
to the eigenfunction $\sin(x-a)$, providing the relation
\begin{equation}\label{fred}
m+\lambda_0 = \frac{\lambda_1}{4} \pi \sin a.
\end{equation}
\begin{lemma}
The Lagrange multiplier $\lambda_1$ is negative.
\end{lemma}
\begin{proof}
We first study the case where $(a,b)$ has length $\ell>\pi$.
Let us  analyse equation \eqref{gr1p}.
We recall that $m+\lambda_0>0$ by \eqref{m+l0};  for $\ell>\pi$, both
terms
$2\tan(\ell/2) -\ell$ and $1+\ell/\sin \ell$ are negative. If
$\lambda_1\geq 0$,
the integral is negative: this is a  contradiction with the sign of $u$ on
$(a,b)$.
\\
In the case of a nodal domain of length $\pi,$ one has   $\lambda_1<0$ by
equation (\ref{fred}), since $m+\lambda_0>0$ and $\sin a<0$.
\end{proof}
\begin{lemma}\label{majol1}
The Lagrange multiplier $\lambda_1$  satisfies
\begin{equation}\label{majl1}
\left|\frac{\lambda_1}{2}\right| \leq 2\sin\left(\frac{\ell_-}{2}\right)\frac{m}{\pi}
= \frac{2}{\ell_-} \sin \frac{\ell_-}{2} (m+\lambda_0) <
m+\lambda_0\,,
\end{equation}
where $\ell_-$ is the measure of $\{x: u(x)<0\}$.
\end{lemma}
\begin{proof}
Let us introduce
the two numbers:
$$\ell_-^b=|\{t>b,u(t)<0\}|,\ \ell_-^a=|\{t<a,u(t)<0\}|\,.$$
Obviously $\ell_-^a+\ell_-^b=\ell_-$.
By \eqref{m+l0},  
$m+\lambda_0 =\displaystyle \frac{m\ell_-}{\pi}$.
Now, since $\lambda_1<0$ by the previous lemma,
$$\left|\frac{\lambda_1}{2}\right|=\frac{m}{2\pi} \int_{-\pi}^\pi sign(u)
\cos t dt$$
and
$$\int_{-\pi}^\pi sign(u) \cos t dt=\int_{-\pi}^a sign(u) \cos t dt+\int_{a}^b
 \cos t dt+\int_{b}^\pi sign(u) \cos t dt.$$
By the bathtub principle (see \cite{lieb-loss}), the value of $\displaystyle
 \int_{-\pi}^a
sign(u) \cos t dt$
is maximum when we choose $sign(u)=-1$  on the left, namely on $(-\pi,-\pi+\ell_-^b]$
(because cos is increasing on $[-\pi,a]$) and similarly for the last
integral.
Therefore, we get 
\begin{equation}\label{gr3}
\left|\frac{\pi\lambda_1}{m}\right|\leq -\int_{-\pi}^{-\pi+\ell_-^b} \cos
t dt+
\int_{-\pi+\ell_-^b}^{\pi+\ell_-^a}  \cos t dt-\int_{\pi-\ell_-^a}^\pi \cos
t dt=
2(\sin \ell_-^a + \sin \ell_-^b).
\end{equation}
Since 
$$\sin \ell_-^a + \sin \ell_-^b =2\sin\left(\frac{\ell_-}{2}\right)\cos\left(\frac{\ell_-^a-\ell_-^b}{2}\right)\leq
2\sin\left(\frac{\ell_-}{2}\right)\leq \ell_-\,,
$$
we finally get estimate \eqref{majl1}, using \eqref{m+l0} and \eqref{gr3}.
\end{proof}

Let us introduce the following positive quantity :
\begin{equation}\label{defA}
A=\dfrac{|\lambda_1/2|}{m+\lambda_0}\,.
\end{equation}
Our strategy to get a contradiction is based on the following
\begin{proposition}\label{prop2pi}
If $A <\frac2{\pi}$ or
$A\cos(\frac{a+b}{2}) < \frac2{\pi}$, then we cannot have a nodal interval of length $\ell \geq \pi$.
\end{proposition}
\begin{proof}
Let us start with the case $b-a=\ell=\pi.$ In that case we have $A=\displaystyle \frac{2}{\pi |\sin
a|}$ by \eqref{fred}. Therefore the assumption
$A< \frac2{\pi}$ provides immediately a contradiction. In the same way, if $A\cos(\frac{a+b}{2}) < \frac2{\pi}$ we deduce
$|\sin a|> \cos(\frac{a+b}{2}) = \cos(\frac{a+a+\pi}{2})=-\sin a$ that is also a contradiction.

\medskip
Now, let us assume that $b-a=\ell > \pi.$ We use the following claim: the function
$g:\ell\mapsto \ell-2\tan(\ell/2) + \frac{4}{\pi} \sin(\ell/2) [1+\ell/\sin \ell]$
is positive  on $(\pi, 2\pi)$.\\
Indeed $g$ is positive if and only if $k(t)=t\cos t -\sin t+\frac{1}{\pi}[2t+\sin(2t)]$
is negative on $\left(\frac{\pi}{2},\pi\right)$. Observe that $k(\frac{\pi}{2})=0$
and $k(\pi)<0$.
Now, the derivative of $k$, $k'(t)=-t\sin(t)+\frac{2}{\pi}[1+\cos(2t)]$,
is the difference between two functions which intersect in only one point
$t_0$. Since $k'$ is negative near $\frac{\pi}{2}$ and positive near $\pi$,
$k$ is minimal at $t_0$ and therefore $k<0$ on $\left(\frac{\pi}{2},\pi\right)$.\\
We are able to get a contradiction by using the expression \eqref{gr1p} of the integral of $u$ on the interval $(a,b)$, that can be written as
$$
\int_{a}^{b} u(x) dx= (m+\lambda_0)\left(2\tan \frac{\ell}{2} -\ell
-2 A \cos \frac{a+b}{2}\ \sin \frac{\ell}{2} \left[1+\frac{\ell}{\sin \ell}\right]\right)\,.
$$
Therefore, if $A <\frac2{\pi}$ or $A\cos(\frac{a+b}{2}) < \frac2{\pi}$, we obtain $\int_{a}^{b} u(x) dx \leq - g(\ell)\leq 0$ 
(note that $1+\ell/\sin\ell <0$ for $\ell>\pi$). Thus we have the desired contradiction.
\end{proof}

We are now going to find some estimates on $A$, in  order to apply Proposition \ref{prop2pi}. This is quite technical and for that reason, we postpone
all these computations to the Appendix.
After proving an estimate on $\ell_-$ (see Proposition \ref{lemma155}), we distinguish the cases where $u$ has at least 6 nodal domains (see Propositions \ref{6nodal_domains_first_case} and \ref{6nodal_domains_second_case}) 
and $u$ has exactly 4 nodal domains (see Proposition \ref{prop_4nodal_domains}).

%%%%%%%%%%%%%%%%%%%%%%%%%%%%%%%%%%%%%%%%
%%%%%%%%%%%%%%%%%%%%%%%%%%%%%%%%%%%
\section{The Lagrange multipliers are zero and the nodal domains have same
length}
Now we enter into the heart of the paper. We are going  to prove that the
Lagrange multipliers $\lambda_0$ and $\lambda_1$ are zero 
(we already know that $\lambda_2=0$) and that 
all the nodal domains have the same length. For that purpose, we will use
the relation
\eqref{derniere_relation} on different intervals. 
\begin{theorem}\label{mainthm1}
The Lagrange multipliers $\lambda_0, \lambda_1$ are equal to zero and all
the nodal intervals have the same length.
\end{theorem}
The proof will be done in two main steps. 
First, we prove that $\lambda_1= 0$ in Proposition \ref{proplambda_1}. Then,
we prove that $\lambda_0=0$  and the nodal domains have same length in Proposition
\ref{proplambda_0}.
\\
Let us first introduce some notations and give a preliminary lemma.
Let $I_k=[a_{k-1},a_k]$ and $I_{k+1}=[a_k,a_{k+1}]$ be two consecutive intervals
of length respectively $\ell_k,\ell_{k+1}$.
We introduce:
$$A(I_k,I_{k+1})=\frac{\ell_k}{\sin\ell_k} \sin a_{k-1} - \frac{\ell_{k+1}}{\sin\ell_{k+1}}
\sin a_{k+1}.$$
Note that, using $a_{k-1}=a_k-\ell_k$ and $a_{k+1}=a_k+\ell_{k+1}$ we can
also write
\begin{equation}\label{eqA}
A(I_k,I_{k+1})= \left(\frac{\ell_k}{\tan\ell_k} - \frac{\ell_{k+1}}{\tan\ell_{k+1}}\right)
\sin a_k -(\ell_k+\ell_{k+1})\cos a_k .
\end{equation}
\\
\begin{lemma}\label{lemma33}
There exist three consecutive intervals, say $I_{j}, I_{j+1}, I_{j+2}$,
such that 
$A(I_{j}, I_{j+1})\geq 0$, $A(I_{j+1}, I_{j+2})\geq 0$
and  there exist three consecutive intervals, say $I_{i}, I_{i+1}, I_{i+2}$,
such that $A(I_{i}, I_{i+1})< 0$, $A(I_{i+1}, I_{i+2})< 0$.
\end{lemma}
\begin{proof}
Let us consider $I_i=(a_{i-1},a_i), I_{i+1}=(a_i,a_{i+1}), I_{i+2}=(a_{i+1},a_{i+2})$,
with $a_i<0<a_{i+1}$. Without loss of generality we can assume that 
$I_i\cup I_{i+1}\cup I_{i+2}\subset [-\pi,\pi]$ (up to consider $u(-x)$ instead
of $u(x)$). 
Since $\sin a_{i-1}<0$ and $\sin a_{i+1}>0$
$$
A(I_i,I_{i+1})=\frac{\ell_i}{\sin \ell_i} \sin a_{i-1} - \frac{\ell_{i+1}}{\sin
\ell_{i+1}} \sin a_{i+1} <0\,.
$$
Since $\sin a_i<0$ and $\sin a_{i+2}>0$
$$
A(I_{i+1},I_{i+2})=\frac{\ell_{i+1}}{\sin \ell_{i+1}} \sin a_i - \frac{\ell_{i+2}}{\sin
\ell_{i+2}} \sin a_{i+2}<0\,.
$$
Let us consider $I_j=(a_{j-1},a_{j}), I_{j+1}=(a_{j},a_{j+1}), I_{j+2}=(a_{j+1},a_{j+2})$,
with $a_{j}<-\pi<a_{j+1}$.
 Assume that $I_j\cup I_{j+1}\cup I_{j+2}\subset [-2\pi,0]$.
Since $\sin a_{j-1}>0$ and $\sin a_{j+1}<0$
$$
A(I_j,I_{j+1})=\frac{\ell_j}{\sin \ell_j} \sin a_{j-1} - \frac{\ell_{j+1}}{\sin
\ell_{j+1}} \sin a_{j+1} >0\,.
$$
Since $\sin a_{j}>0$ and $\sin a_{j+2}<0$
$$
A(I_{j+1},I_{j+2})=\frac{\ell_{j+1}}{\sin \ell_{j+1}} \sin a_{j} - \frac{\ell_{j+2}}{\sin
\ell_{j+2}} \sin a_{j+2} > 0\,.
$$
\end{proof}
\begin{proposition}\label{proplambda_1}
The Lagrange multiplier $\lambda_1$ is zero.
\end{proposition}
\begin{proof}
Let $I_k=[a_{k-1},a_k]$, $I_{k+1}=[a_{k},a_{k+1}]$, $I_{k+2}=[a_{k+1},a_{k+2}]$,
$I_{k+3}=[a_{k+2},a_{k+3}]$ 
four consecutive intervals of lengths 
$\ell_{k},\ell_{k+1},\ell_{k+2},\ell_{k+3}$, respectively. We assume that
$u$ is alternatively
 positive, negative, positive and negative.
In Section \ref{preliminaries} we have seen that (see \eqref{derniere_relation})
\begin{equation}\label{E1}
\lambda_0 \sin \frac{\ell_{k}+\ell_{k+1}}{2} - m \sin \frac{\ell_{k+1}-\ell_{k}}{2}
=
\frac{\lambda_1}{2}
\cos \frac{\ell_{k}}{2} \cos \frac{\ell_{k+1}}{2} A(I_{k},I_{k+1})\,.
\end{equation}
We can reproduce this identity for the other intervals:
\begin{equation}\label{E2}
\lambda_0 \sin \frac{\ell_{k+1}+\ell_{k+2}}{2} + m \sin \frac{\ell_{k+2}-\ell_{k+1}}{2}
=
\frac{\lambda_1}{2}
\cos \frac{\ell_{k+1}}{2} \cos \frac{\ell_{k+2}}{2} A(I_{k+1},I_{k+2})
\end{equation}
\begin{equation}\label{E3}
\lambda_0 \sin \frac{\ell_{k+2}+\ell_{k+3}}{2} - m \sin \frac{\ell_{k+3}-\ell_{k+2}}{2}
=
\frac{\lambda_1}{2}
\cos \frac{\ell_{k+2}}{2} \cos \frac{\ell_{k+3}}{2} A(I_{k+2},I_{k+3})\,.
\end{equation}
Assume by contradiction that $\lambda_1\neq 0$. We divide the proof into
three cases, according to the lengths of the nodal intervals.
\begin{enumerate}
\item
Let us assume that $\ell_{k}\not= \ell_{k+2}$ and $\ell_{k+1}\not= \ell_{k+3}$.
Equations \eqref{E1}, \eqref{E2} can be seen as a system in $\lambda_0$ and
$m$ from which
we get
$$\lambda_0=\frac{\lambda_1}{2} \cos \frac{\ell_{k+1}}{2}
\dfrac{\cos \frac{\ell_{k}}{2} \sin \frac{\ell_{k+2}-\ell_{k+1}}{2} A(I_{k},I_{k+1})
+ \cos
\frac{\ell_{k+2}}{2} \sin \frac{\ell_{k+1}-\ell_{k}}{2} 
A(I_{k+1},I_{k+2}) }{\sin \ell_{k+1} \sin \frac{\ell_{k+2}-\ell_{k}}{2}}
$$
$$
m=\frac{\lambda_1}{2} \cos \frac{\ell_{k+1}}{2}
\dfrac{-\cos \frac{\ell_{k}}{2} \sin \frac{\ell_{k+2}+\ell_{k+1}}{2} A(I_{k},I_{k+1})
+ \cos
\frac{\ell_{k+2}}{2} \sin \frac{\ell_{k+1}+\ell_{k}}{2} 
A(I_{k+1},I_{k+2}) }{\sin \ell_{k+1} \sin \frac{\ell_{k+2}-\ell_{k}}{2}}\,.
$$
We observe that
\begin{equation}\label{somme}
\lambda_0+m=\frac{\lambda_1}{2} \dfrac{\cos \frac{\ell_{k}}{2} \cos \frac{\ell_{k+2}}{2}}{\sin
\frac{\ell_{k+2}-\ell_{k}}{2}} [A(I_{k+1},I_{k+2}) - A(I_{k},I_{k+1})]\,.
\end{equation}
Similarly, if one chooses  equations \eqref{E2}, \eqref{E3} to solve with
respect to
$\lambda_0$, $m$, he gets
$$\lambda_0+m=
\frac{\lambda_1}{2} \dfrac{\cos \frac{\ell_{k+2}}{2}}{\sin \frac{\ell_{k+2}}{2}
\sin \frac{\ell_{k+3}-\ell_{k+1}}{2}}\times
$$
\begin{equation}\label{sommebis}
\times
\left[\cos \frac{\ell_{k+1}}{2} \sin \frac{\ell_{k+3}}{2}A(I_{k+1},I_{k+2})
 -\cos \frac{\ell_{k+3}}{2} \sin \frac{\ell_2}{2} A(I_{k+2},I_{k+3})\right].
\end{equation}
We now set $\displaystyle C=\frac{2(m+\lambda_0)}{\lambda_1}$.  We use \eqref{somme}
to get $A(I_{k+1},I_{k+2})$ in terms
of $A(I_{k},I_{k+1})$:
\begin{equation}\label{rel1}
A(I_{k+1},I_{k+2})=A(I_{k},I_{k+1})+\frac{C \sin \frac{\ell_{k+2}-\ell_{k}}{2}}{\cos
\frac{\ell_{k}}{2}
\cos \frac{\ell_{k+2}}{2}}\,.
\end{equation}
We also use \eqref{sommebis} to get $A(I_{k+1},I_{k+2})$ in terms of $A(I_{k+2},I_{k+3})$:
\begin{equation}\label{rel2}
A(I_{k+1},I_{k+2})= \frac{\tan \frac{\ell_{k+1}}{2}}{\tan \frac{\ell_{k+3}}{2}}
A(I_{k+2},I_{k+3})+
\frac{C \tan \frac{\ell_{k+2}}{2} \sin \frac{\ell_{k+3}-\ell_{k+1}}{2}}{\sin
\frac{\ell_{k+3}}{2}
\cos \frac{\ell_{k+1}}{2}}\,.
\end{equation}
Since  $m$ is  non-zero, the $3\times 3$ determinant of the system in $(m,\lambda_0,\lambda_1)$
given by 
equations \eqref{E1}, \eqref{E2}, \eqref{E3}
has to be equal to zero. Now, the computation of  this determinant with respect
to
its third column
gives the following equality after some simplification:
\begin{equation}\label{det3}
\begin{array}{l}
A(I_{k},I_{k+1}) \cos \frac{\ell_{k}}{2} \cos \frac{\ell_{k+1}}{2} \sin\ell_{k+2}
\sin \frac{\ell_{k+1}-\ell_{k+3}}{2}
\\
+ A(I_{k+2},I_{k+3}) \cos \frac{\ell_{k+2}}{2} \cos \frac{\ell_{k+3}}{2}
\sin\ell_{k+1} \sin
\frac{\ell_{k+2}-\ell_{k}}{2} \\
- A(I_{k+1},I_{k+2}) \cos \frac{\ell_{k+1}}{2} \cos \frac{\ell_{k+2}}{2}\times
\\
\times\left(\sin \frac{\ell_{k+1}-\ell_{k+3}}{2}
\sin \frac{\ell_{k}+\ell_{k+2}}{2} + \sin \frac{\ell_{k+2}-\ell_{k}}{2} \sin
\frac{\ell_{k+1}+\ell_{k+3}}{2}\right)
= 0.
\end{array}
\end{equation}
Now we replace $A(I_{k+1},I_{k+2})$ in \eqref{det3} by using both \eqref{rel1}
(for
the first term), \eqref{rel2} (for the second) and we get, after use of trigonometric
formulae
\begin{equation}
\begin{array}{l}
0=A(I_{k},I_{k+1}) \cos \frac{\ell_{k+1}}{2} \cos \frac{\ell_{k+2}}{2} \sin
\frac{\ell_{k+1}-\ell_{k+3}}{2}
\sin \frac{\ell_{k+2}-\ell_{k}}{2}+\\
+ A(I_{k+2},I_{k+3}) \cos \frac{\ell_{k+2}}{2} \cos \frac{\ell_{k+3}}{2}
\sin \frac{\ell_{k+3}-\ell_{k+1}}{2}
\sin \frac{\ell_{k+2}-\ell_{k}}{2}\frac{\sin \frac{\ell_{k+1}}{2}}{\sin \frac{\ell_{k+3}}{2}}\\
- C \frac{\sin \frac{\ell_{k+1}-\ell_{k+3}}{2} \sin \frac{\ell_{k+2}-\ell_{k}}{2}}{\sin
\frac{\ell_{k+3}}{2}
\cos \frac{\ell_{k}}{2}} \left[\cos \frac{\ell_{k+1}}{2} \sin \frac{\ell_{k+3}}{2}
\sin \frac{\ell_{k}+\ell_{k+2}}{2} + \cos \frac{\ell_{k}}{2} \sin \frac{\ell_{k+2}}{2}
\sin \frac{\ell_{k+1}+\ell_{k+3}}{2}\right]\,.
\end{array}
\end{equation}   
Simplifying by   $\sin \frac{\ell_{k+1}-\ell_{k+3}}{2} \sin \frac{\ell_{k+2}-\ell_{k}}{2}$
(this is possible since we are assuming
$\ell_{k+2}\not= \ell_{k}$ and $\ell_{k+3}\not= \ell_{k+1}$) 
 we finally get
\begin{eqnarray}\label{determinant}
\frac{C}{\sin \frac{\ell_{k+3}}{2} \cos \frac{\ell_{k}}{2}}\times
\\\nonumber\times \left[\cos \frac{\ell_{k+1}}{2}
\sin \frac{\ell_{k+3}}{2} \sin \frac{\ell_{k}+\ell_{k+2}}{2} + \cos \frac{\ell_{k}}{2}
\sin \frac{\ell_{k+2}}{2} \sin \frac{\ell_{k+1}+\ell_{k+3}}{2}\right]
\\ \nonumber =
A(I_{k},I_{k+1}) \cos \frac{\ell_{k+1}}{2} \cos \frac{\ell_{k+2}}{2}
- A(I_{k+2},I_{k+3}) \frac{\cos \frac{\ell_{k+2}}{2} \cos \frac{\ell_{k+3}}{2}\sin
\frac{\ell_{k+1}}{2}}{\sin
\frac{\ell_{k+3}}{2}}\,.
\end{eqnarray}  
Note that $C$ has  the
same sign as $\lambda_1$, since by definition of $\lambda_0$ (see section
\ref{preliminaries}), $\lambda_0+m>0$.
Moreover, in equation \eqref{determinant}, the coefficients of $C$,  $A(I_{k},I_{k+1})$
and $-A(I_{k+2},I_{k+3})$ are all positive, since the length of each nodal domain is less than $\pi$, as we have seen in Section \ref{lengthpi}.
\\
Now we claim that we can choose four consecutive intervals such that 
\begin{itemize}
\item $A(I_{k},I_{k+1})$ is positive and $A(I_{k+2},I_{k+3})$ is negative
(with $u$ positive on $I_k$).
\end{itemize}
and we can choose four intervals such that 
\begin{itemize}
\item $A(I_{j},I_{j+1})$ is negative and $A(I_{j+2},I_{j+3})$ is positive
(with $u$ positive on $I_j$).
\end{itemize}
If this claim is true, we get a contradiction since \eqref{determinant} would
show that $C$ (and then $\lambda_1$) is both positive and negative.
 To prove our claim, we set 
$$
\mathcal{I}_- = \{(I_k, I_{k+1}) : A(I_k,I_{k+1})<0\},\,\,\,\,\, \mathcal{I}_+
= \{(I_k, I_{k+1}) : A(I_k,I_{k+1})\geq 0\}\,.
$$
We have seen in Lemma \ref{lemma33} that both $\mathcal{I}_-$ and $\mathcal{I}_+$
contain pairs of consecutive intervals (or triplet of intervals).
Let us now consider the last triplet of intervals for which $A<0$. 
%We can assume that $u$ is positive on the first interval and negative on
the second one.
Let $I_{k-1}, I_k, I_{k+1}$ be these three intervals. Therefore $A(I_{k+1},
I_{k+2})\geq 0$.
If $u>0$ on $I_{k-1}$ we are done, because we can consider $(I_{k-1},I_k),
(I_{k+1},I_{k+2})$.
If $u$ is negative on $I_{k-1}$ we have $u$ negative on $I_{k-1}, I_{k+1},
I_{k+3} \ldots$ and positive on
$I_{k}, I_{k+2}, I_{k+4} \ldots$.
We can consider $(I_{k}, I_{k+1})$ for which $A<0$.
If $A(I_{k+2}, I_{k+3})\geq 0$ we are done.
If $A(I_{k+2}, I_{k+3})\leq 0$, then $A(I_{k+3, I_{k+4}})\geq 0$ (otherwise
the last triplet in $\mathcal{I}_-$ would be $I_{k+2},I_{k+3};I_{k+4}$).
If $A(I_{k+4}, I_{k+5})\geq 0$ we are done.
If $A(I_{k+4}, I_{k+5})\leq 0$, then $A(I_{k+5, I_{k+6}})\geq 0$....after
some steps we will get necessarily
four consecutive intervals $I_{m-2},I_{m-1}, I_m, I_{m+1}$ (with $u$ positive
on $I_{m-2}$)
such that $A(I_{m-2},I_{m-1})<0$, $A(I_{m}, I_{m+1})\geq 0$ (because we have
to stop before the first triplet of $\mathcal{I}_+$).
Therefore, we have proved the first part of our claim. The second part is
proved exactly in the same way, starting from the last triplet in $\mathcal{I}_+$.
\\
In conclusion, we get a contradiction. 
\item
 Assume  $\ell_k=\ell_{k+2}$. Since the left-hand sides of equations \eqref{E1}
and \eqref{E2}
 coincide,  the right-hand sides are equal that implies necessarily (since
we are assuming $\lambda_1\neq 0$)
$$ \frac{\ell_k}{\sin\ell_k}(\sin a_{k-1}+\sin
a_{k+2})=
\frac{\ell_{k+1}}{\sin\ell_{k+1}}(\sin a_{k}+\sin a_{k+1}).$$
Now, we observe that since $\ell_k=\ell_{k+2}$, one has  $(a_{k-1}+a_{k+2})/2=(a_{k}+a_{k+1})/2$.
Replacing
$\sin a_{k-1}+\sin a_{k+2}$ by $2\sin (a_{k-1}+a_{k+2})/2 \cos (a_{k+2}-a_{k-1})/2$
and $\sin a_{k}+\sin a_{k+1}$ by $2\sin
(a_{k}+a_{k+1})/2 \cos (a_{k+1}-a_{k})/2$ the above equality gives
\begin{equation}\label{equa}
\frac{\ell_k}{\sin\ell_k} \cos(\ell_k+\frac{\ell_{k+1}}{2}) = \frac{\ell_{k+1}}{\sin\ell_{k+1}}
\cos\frac{\ell_{k+1}}{2} =\dfrac{\ell_{k+1}}{2\sin\frac{\ell_{k+1}}{2}}\,.
\end{equation}
Assuming $\ell_k$ fixed, we can study the function 
$$g: x\mapsto \ell_k\sin x \cos(\ell_k+x) - x\sin\ell_k.$$
Since $g'(x)$ is  negative
and  $g(0)=0$,  it is not possible to find $\ell_{k+1}>0$ such
that \eqref{equa} holds.
Therefore we have a contradiction. 
\item
Assuming 
$\ell_{k+1}=\ell_{k+3}$
we get a  contradiction in the same way as in the previous case ($\ell_k=\ell_{k+2}$).

\end{enumerate}
Therefore, we conclude that necessarily $\lambda_1=0$.
\end{proof}

To finish the proof of Theorem \ref{mainthm1}, we need the following proposition.
\begin{proposition}\label{proplambda_0}
The Lagrange multiplier 
$\lambda_0$ is zero and the nodal domains have same length.
\end{proposition}
\begin{proof}
Since $\lambda_1=0$ by the previous proposition, from section \ref{preliminaries}
we deduce that
\begin{equation}\label{primalambda0}
\lambda_0\sin\left(\frac{\ell_k+\ell_{k+1}}{2}\right) - m \sin\left(\frac{\ell_{k+1}-\ell_k}{2}\right)=0\,,
\end{equation}
$$
\lambda_0\sin\left(\frac{\ell_{k+1}+\ell_{k+2}}{2}\right) + m \sin\left(\frac{\ell_{k+2}-\ell_{k+1}}{2}\right)=0\,.
$$
The determinant of this homogeneous system is zero, as $m\neq 0$. This means
$$
\sin\left(\frac{\ell_{k+1}-\ell_k}{2}\right)\sin\left(\frac{\ell_{k+1}+\ell_{k+2}}{2}\right)=
\sin\left(\frac{\ell_{k+1}-\ell_{k+2}}{2}\right)\sin\left(\frac{\ell_k+\ell_{k+1}}{2}\right)
$$
that is,
$$
\cos\left(\frac{2\ell_{k+1}-\ell_k+\ell_{k+2}}{2}\right)=
\cos\left(\frac{2\ell_{k+1}+\ell_k-\ell_{k+2}}{2}\right) 
$$
which implies
$\ell_k-\ell_{k+2}=-\ell_k+\ell_{k+2}$, that is, $\ell_k=\ell_{k+2}$.
With the same argument on $\ell_{k+1}, \ell_{k+2}, \ell_{k+3}$ we find $\ell_{k+1}=\ell_{k+3}$.
\\
Therefore all the intervals where $u$ is positive have the same length, say
$\ell_1$; 
all the intervals where $u$ is negative have the same length, say $\ell_2$.
The sum of these lengths give $n(\ell_1+\ell_2)=2\pi$.
\\
On the other hand, $\displaystyle \lambda_0=-\frac{m}{2\pi}\int_0^{2\pi}
sign(u)=-\frac{m}{2\pi}n(\ell_1-\ell_2)$
which gives us
\begin{equation}\label{elle12lambda0}
\ell_2-\ell_1 = \frac{2\pi \lambda_0}{mn}\,.
\end{equation}
If we replace this equality in (\ref{primalambda0}), we have
$$
\lambda_0\sin\left(\frac{\pi}{n}\right)
=m\sin\left(\frac{\pi \lambda_0}{m n}\right)\,.
$$
We now study the function $\displaystyle f(x)=m\sin\left(\frac{\pi x}{m n}\right)
- x\sin\left(\frac{\pi}{n}\right)$, for $x\in [0,m)$ (recall that $\lambda_0<m$).
Since $f(0)=0=f(m)$, $f'(0)>0$, $f'(m)<0$ and $f''(x)<0$, we deduce that
the only zero on $f$ is zero. This means that
$f(\lambda_0)=0$ if and only if $\lambda_0=0$. This argument proves that
$\lambda_0=0$.
We deduce from \eqref{elle12lambda0} that $\ell_1=\ell_2$. 
\end{proof}

%%%%%%%%%%%%%%%%%%%%%%%%%%%%%%%%%%%%%
%%%%%%%%%%%%%%%%%%%%%%%%%%%%%%%%%%%%%

\section{Conclusion}\label{Conclusion}
We are now in position to prove Theorem \ref{mainthm}.
We have seen in the previous section that all 
the nodal intervals have same length, say $\ell<\pi$, and the Lagrange multipliers
$\lambda_0, \lambda_1, \lambda_2$
are all zero.
We recall that if $(a,b)$ is an interval where $u\geq 0$, one has 
$$
u(x)=A_0\cos x +B_0\sin x -m, \,\,\,A_0=m\frac{\cos(\frac{a+b}{2})}{\cos
\frac{\ell}{2}}\,,
B_0=m\frac{\sin(\frac{a+b}{2})}{\cos
\frac{\ell}{2}}\,;
$$
if $(b,c)$ is an interval  where $u\leq 0$, one has 
$$
u(x)=A_1\cos x +B_1\sin x +m, \,\,\,A_1=- m\frac{\cos(\frac{b+c}{2})}{\cos
\frac{\ell}{2}}\,,
B_1= - m\frac{\sin(\frac{b+c}{2})}{\cos
\frac{\ell}{2}} 
$$
(see subsection \ref{subgenral}).
The solution of the system
$$
\left\{
\begin{array}{l}
u(x)=A_0\cos x +B_0\sin x -m=0
\\
u(x)=A_1\cos x +B_1\sin x +m=0
\end{array}
\right.
$$
is
$$
\left\{
\begin{array}{l}
\cos(x-\varphi_0)=\cos \frac{\ell}{2},\,\,\, \tan(\varphi_0)=\tan\left(\frac{a+b}{2}\right)
\\
\cos(x-\varphi_1)=\cos \frac{\ell}{2},\,\,\, \tan(\varphi_1)=\tan\left(\frac{b+c}{2}\right)
\end{array}
\right.
$$
with $x-\varphi_0=\pm (\pi +x-\varphi_1)+2k\pi$. Since $x=a+\ell$, the only
possible solution is $\ell=\frac{\pi}{2}$.
Therefore $u(x)$ is symmetric with respect to $a$, and 
$$
u(x)=
\left\{
\begin{array}{l}
m\cos x +m\sin x - m, x\in \left[a,a+\frac{\pi}{2}\right]
\\
-m\cos x -m\sin x + m, x\in \left[a+\frac{\pi}{2},a+\pi\right]\,.
\end{array}
\right.
$$
We now compute $m$ defined in (\ref{opepl}). Recalling that $\displaystyle
\int_0^1 |u|=1$ (see (\ref{normaL1=1})), we have
$$\displaystyle \left[\int_0^{2\pi}|u|\right]^2= 16\left[\int_a^{a+\frac{\pi}{2}}
|u|\right]^2=16 m^2\left(2-\frac{\pi}{2}\right)^2=1\,.
$$
Therefore $\displaystyle m=\frac{1}{2(4-\pi)}.$
%\approx 0.58$ and $\displaystyle \frac{\pi}{4}m\approx 0.457$.

\section{Motivations and final remarks}
The minimization of the functional in (\ref{opepl}) is motivated by a shape optimization problem and more precisely
from a quantitative isoperimetric inequality.
Indeed, for any open bounded set of $\mathbb{R}^n$, let us introduce the {\it isoperimetric
deficit}:
\begin{equation}\label{def-delta}
\delta(\Omega)=\frac{P(\Omega)-P(B)}{P(B)}\,,
\end{equation}
where $ |B|=|\Omega|$. 
Let the {\it barycentric asymmetry} be defined by:
$$
\lambda_0(\Omega)=\frac{|\Omega \Delta B_{x^G}|}{|\Omega|}
$$
where $B_{x^G}$ is the ball centered at the barycentre $\displaystyle {x^G}=\frac{1}{|\Omega|}\int_{\Omega}
x\,dx$ of $\Omega$
and such that $|\Omega|=|B_{x^G}|$.  
Fuglede proved in \cite{Fu93Geometriae}
that there exists a positive constant (depending only on the
dimension $n$) such that
\begin{equation}\label{Fuglede_convex}
\delta(\Omega)\geq C(n)\,\lambda_0^2(\Omega),\quad\textnormal{for any convex subsets
$\Omega$ of $\mathbb{R}^n$}.
\end{equation}
Now, the constant $C(n)$ is unknown (as it is the case in most quantitative
inequalities
like \eqref{Fuglede_convex}) and it would be interesting to find the best
constant.
This leads to consider the minimization of the ratio
$$\displaystyle \ratioo(\Omega)=\frac{\delta(\Omega)}{\lambda_0^2(\Omega)}$$
among convex compact sets in the plane, in particular. In the study of this
minimization problem,
one is led to exclude sequences converging to the ball in the Hausdorff metric.
The strategy is to prove that
on these sequences
$\ratioo$ is greater than
0.406 which is the value of $\ratioo(S)$ where $S$ is a precise set with
the shape of a stadium, as computed  in \cite{AFN}.

If a convex planar set $E$ has barycenter in 0, it can be parametrized in
polar coordinates
with respect to 0, as
\begin{equation}\label{polar_coord}
E=\{y\in \mathbb{R}^2: y=tx(1+u(x)), x\in \mathbb{S}^1, t\in [0,1]\}\,,
\end{equation}
where $u$ is a Lipschitz periodic function.
Then the shape functional $\ratioo(E)$  can be written as a functional $H$
of the function $u$ describing
$E$, as follows :
\begin{equation}\label{def-J}
\ratioo(E)=H(u)=\frac{\pi}{2} \frac{\displaystyle\int_{-\pi}^{\pi}\left[\sqrt{(1+u)^2+u'(\theta)^2}-1\right]d\theta}{\left[\frac
12 \displaystyle\int_{-\pi}^{\pi} |(1+u)^2-1|d\theta\right]^2}.
\end{equation}
The constraints of area (fixed equal to $\pi$ without loss of generality)
and barycentre in $0$ read  in terms of a periodic $u\in H^1(-\pi,\pi)$ as:
\begin{itemize}
\item [(NL1)]
$\displaystyle\frac{1}{2\pi}\int_{-\pi}^{\pi} (1+u)^2d\theta=1$;
\item [(NL2)]
$\displaystyle\int_{-\pi}^{\pi}\cos(\theta)[1+u(\theta)]^3d\theta=0$;
\item [(NL3)]
$\displaystyle\int_{-\pi}^{\pi}\sin(\theta)[1+u(\theta)]^3d\theta=0$.
\end{itemize}
The computation of the minimum of $H$, under the constraints  
(NL1), (NL2) and (NL3), seems
very difficult. However, for
sequences of sets converging to the ball in the Hausdorff metric,
the limit of 
$$m_\e:=\inf \{H(u), \|u\|_{L^\infty}=\ep,\;u\in H^1(-\pi,\pi)\,\, \text{periodic,
satisfying \,(NL1), (NL2), (NL3)}\}$$
as $\varepsilon\to 0$, equals the limit of the shape  functional $\ratioo$
for these sequences. 
Thus,
a possible strategy consists in estimating from below
the minimum of $H$
 by a simpler functional, namely its linearization.
Define
$$
m=\inf_{u\in\mathcal{W}}\frac{ \displaystyle \int_{-\pi}^{\pi}[(u')^2-u^2]d\theta}{\displaystyle
\left[\int_{-\pi}^{\pi} |u| d\theta\right]^2}
$$
where $\mathcal{W}$ is the space of periodic $H^1(0,2\pi)$ functions satisfying
the
constraints:
\begin{itemize}
\item[(L1)]
$\displaystyle \int_{-\pi}^{\pi} u\,d\theta=0$
\item[(L2)]
$\displaystyle\int_{-\pi}^{\pi} u \cos(\theta)\,d\theta=0$
\item[(L3)]
$\displaystyle\int_{-\pi}^{\pi} u \sin(\theta)\,d\theta=0$.
\end{itemize}
In \cite{BCH-barycenter} we proved that 
\begin{equation}\label{conclusion}
\liminf_{\e \to 0} m_\e \geq \frac{\pi}{4}m.
\end{equation}
The value of $m$ found in Theorem \ref{mainthm}
allows us to conclude that
$$
\liminf_{\e \to 0} m_\e \geq \frac{\pi}{4}m>0.406\,.
$$

\begin{remark}
We observe that one can easily get an estimate from below of $m$ by using the Cauchy-Schwarz
inequality 
$$\left(\int_{-\pi}^{\pi} |u| d\theta\right)^2\leq 2\pi \int_{-\pi}^{\pi}
u^2 d\theta.$$
Then, a Wirtinger-type inequality (or Parseval formula) shows that 
$$m \geq
\inf_{u\in\mathcal{W}}\frac{ \displaystyle \int_{-\pi}^{\pi}[(u')^2-u^2]d\theta}{\displaystyle
2\pi \int_{-\pi}^{\pi} u^2 d\theta} \geq \dfrac{3}{2\pi}.
$$
Unfortunately this estimate on $m$
 is not sufficient to prove the desired inequality
 $
\displaystyle \liminf_{\e \to 0} m_\e >0.406\,.
$

\end{remark}
\begin{remark}
One could be tempted by looking for an approximation of the value of $m$, considering the subset of $\mathcal{W}$ composed by  piecewise affine functions, which are 0 on the same set of zeros as a minimizer $u$. Unfortunately this strategy would give an estimate from above of $m$. Instead, we need an estimate from below for our quantitative isoperimetric inequality.
\end{remark}

%%%%%%%%%%%%%%%%%%

\begin{acknowledgements}
This work was partially supported by the project ANR-18-CE40-0013 SHAPO financed
by the French Agence Nationale de la Recherche (ANR).
We kindly thank the anonimouos referee for his precious remarks and suggestions. 
\end{acknowledgements}

 \section*{Conflict of interest}
 The authors declare that they have no conflict of interest.



%%%%%%%%%%%%%%%%%%%%%%%%%%%%%%%%%%%%%%%%%%%%%%%
%%%%%%%%%%%%%%%%%%%%%%%%%%%%%%%%%%%%%%%%%%%%%%%%
%%%%%%%%%%%%%%%%%%%%%%%%%%%%%%%%%%%%%%%%%%%%%
%%%%%%%%%%%%%%%%%%%%%%%%%%%%%%%%%%%%%%%%%%%%%%%%%
\section{Appendix}
In this section we prove the most technical results of Section \ref{lengthpi}, under the assumptions given at the beginning of that section. 
We recall that our aim is to prove that
 $A<\frac{2}{\pi}$ and thus to get a contradiction, as explained in Proposition \ref{prop2pi}.
\\
Let us denote by $m_j$ the midpoint of a nodal interval $(a_j,b_j)$. 
\begin{lemma}
Let $(a_j,b_j)$ be a negative interval.
If $\displaystyle m_j \notin \left[-\frac{\pi}{2},\frac{\pi}{2}\right]$,
then
\begin{equation}\label{mainn}
A \leq \frac{l_j^2 (2\pi-\ell_-)}{12 \ell_- |\cos(m_j)|}\,.
\end{equation}
For a positive interval $(a_k,b_k)$, whose midpoint $\displaystyle m_k \in
\left(-\frac{\pi}{2},\frac{\pi}{2}\right)$,
one has
\begin{equation}\label{mainp}
A \leq \frac{l_k^2}{12 \cos(m_k)}.
\end{equation}
\end{lemma}
\begin{proof}
For a negative interval 
$(a_j,b_j)$, as soon as $\displaystyle \frac{a_j+b_j}{2} \notin \left[-\frac{\pi}{2},\frac{\pi}{2}\right]$
the fact that $\displaystyle \int_{a_j}^{b_j} u <0$ implies 
$$
\dfrac{|\lambda_1/2|}{m-\lambda_0} \leq \frac{h(\ell_j)}{|\cos(\frac{a_j+b_j}{2})|}\,,
$$
where $h(x)=\dfrac{2\tan(\frac{x}{2})-x}{2\sin(\frac{x}{2})\left(1+\frac{x}{\sin
x}\right)}$.
We claim that for $x\in [0,\pi]$, one has
$h(x)\leq\frac{x^2}{12}
$.
By using \eqref{m+l0}, \eqref{m-l0}, \eqref{ml0} and this  bound on $h$
we have the following estimate on $A$:
$$A \leq \frac{l_j^2 (2\pi-\ell_-)}{12 \ell_- |\cos(\frac{a_j+b_j}{2})|}\,.
$$
We now prove our claim, that is, the bound on $h$.
The statement is equivalent to the positivity of $f(x)=\frac{x^2}{12}(x+\sin
x)-2\sin(x/2)+x\cos(x/2)$ in $[0,\pi]$. Observe that $f(0)=0$.
The result will follow if we prove that the derivative of $f$ is positive,
that  is, $k(x)=3x+2\sin(x)+x\cos(x)-6\sin(x/2)$ is positive. We remark that
$k(0)=0, k(\frac{\pi}{2})>0, k(\pi)>0$.
We will split the analysis into two cases.
\begin{enumerate}
\item
If $x\in [0,\frac{\pi}{2}]$ it is easy to see that $k''<0$ and therefore
$k(x)\geq 0$ in $[0,\frac{\pi}{2}]$.
\item
In the case where $x\in [\frac{\pi}{2},\pi]$, $k$ is decreasing. Indeed it
is easy to see that $k'$ is negative, since $k'$ is convex, $k'(\pi)=0, k'(\frac{\pi}{2})<0$.
\end{enumerate}
In the same way, using  \eqref{gr1p}, we can prove that, on a positive interval
$(a_k,b_k)$ ($m_k$ being its midpoint):
$$
A \leq \frac{l_k^2}{12 \cos(m_k)}\,,
$$
when
$m_k \in (-\frac{\pi}{2},\frac{\pi}{2})$.
\end{proof}
Estimate of $A$ involves the midpoint of nodal domains, as we have just seen. The next lemma gives us important information about nodal domains whose midpoint is between $-\pi/2$ and $\pi/2$ or outside this interval.
\begin{lemma}\label{just_one}
There are no negative intervals on the "left" of $I_0=(a,b)$
whose midpoint is between $-\pi/2$ and $\pi/2$. 
There is at most one negative
interval on the "right" of $I_0$ whose midpoint lies between $-\pi/2$ and
$\pi/2$.
\end{lemma}
\begin{proof}
Since $\displaystyle a\leq -\frac{\pi}{2}$, there are no negative intervals
on the "left" of $I_0$
whose midpoint is between $\displaystyle -\frac{\pi}{2}$ and $\displaystyle
\frac{\pi}{2}$. 
\\
We prove the second statement by contradiction.
If there were two negative intervals, say $I_1$ and $I_3$ with a positive
interval
$I_2=(a_2,b_2)$ between them, its midpoint would satisfy $\displaystyle m_2\leq
\frac{\pi-\ell_2}{2}$.
Moreover, since $\displaystyle b_2\leq \frac{\pi}{2}$,
and $\displaystyle 0\leq b < a_2=b_2-\ell_2\leq \frac{\pi}{2}-\ell_2$, we
infer $\displaystyle \ell_2\leq \frac{\pi}{2}$.
Now using \eqref{mainp}, we see that
$$A\leq \frac{\ell_2^2}{12 \cos(m_2)}\leq \frac{\ell_2^2}{12 \sin(\ell_2/2)}.$$
Now, it is immediate to check that $x\mapsto x^2/\sin(x/2)$ is increasing.
Therefore
the previous inequality would imply for $\ell_2\leq \frac{\pi}{2}$: $A\leq
\pi^2/(24\sqrt{2})
<\frac2{\pi}$ implying $\displaystyle \int_{a_2}^{b_2} u dt <0$, that is,
a contradiction.
\end{proof}

We are now going to find a lower bound for $\ell_-$, the measure of $\{u<0\}$.
We know that 
$$
\displaystyle \int_{\{u>0\}} u(x) dx + \int_{\{u<0\}} u(x)
dx =0\,,$$ 
while
$$\displaystyle \int_{-\pi}^\pi |u(x)| dx= \int_{\{u>0\}} u(x) dx - \int_{\{u<0\}}
u(x) dx = 1\,.
$$ Therefore
\begin{equation}\label{int12}
\int_{\{u>0\}} u(x) dx = - \int_{\{u<0\}} u(x) dx = \frac{1}{2}.
\end{equation}
The following proposition gives a lower bound for $\ell_-$. 
\begin{proposition}\label{lemma155}
$\ell_-\geq  1.55$.
\end{proposition}
%We are now in position to prove the main result of this section.
%\begin{theorem}
%There is no nodal interval of length greater than $\pi$.
%\end{theorem}
%This theorem will be a corollary of several propositions. Since they are quite technical and long, we give here the proof of the simplest ones. The proof of the other ones is given in the Appendix.
\begin{proof}
We start with a simple estimate on $m$ that will be useful in the proof.
Using 
the explicit function given in Theorem \ref{mainthm}
as a test function in the functional  defined by (\ref{opepl}),
we get 
\begin{equation}\label{mdallalto}
m\leq \frac{1}{2(4-\pi)}\,.
\end{equation}
We now prove the estimate of the statement.  We can assume that
the length of all negative intervals is less than $1.6$ (otherwise, if there is a negative interval of length $\ell_j\geq 1.6$, the estimate
$\ell_-\geq  1.55$ is straightforward).
By  Lemma \ref{just_one},
we can split the analysis into
two cases:
\begin{enumerate}
\item no negative interval has its midpoint in $[-\frac{\pi}{2},\frac{\pi}{2}]$;
\item
the midpoint of only one negative interval, say $I_1$, 
belongs to $[-\frac{\pi}{2},\frac{\pi}{2}]$;
\end{enumerate}
that we are now going to develop.
\begin{enumerate}
\item
Assume  that no negative interval has its midpoint in $\displaystyle \left[-\frac{\pi}{2},\frac{\pi}{2}\right]$.
We are going to estimate the sum of 
$
\int_{I_j} (-u(x)) dx
$
on all negative
intervals $I_j$. 
We observe that on all (but possibly one) negative intervals $I_j$ whose midpoint $\displaystyle
m_j \notin \left[-\frac{\pi}{2},\frac{\pi}{2}\right]$, 
we have
\begin{equation}\label{stima_I_j}
\int_{I_j} -u(x) dx \leq (m-\lambda_0)[2\tan(\ell_j/2) -\ell_j]
\end{equation}
by (\ref{gr1n}), since the second term of the right hand side is negative.
Now, it is easy to prove that 
\begin{equation}\label{stima_tg}
2\tan(x/2) -x \leq \frac{0.45 x^3}{4},\quad 0\leq x\leq 0.8\,.
\end{equation}
Therefore inequalities \eqref{stima_I_j} and (\ref{int12}) imply
$$\frac{1}{2}=\sum_j \int_{I_j} (-u(x)) dx \leq \frac{0.45(m-\lambda_0)}{4}
\sum_j \ell_j^3.$$
The maximization of the convex function $t\in \R^+\to t^3$ and the fact
that the sum of the lengths $\ell_j$ of all negative intervals is $\ell_-$
give
$$
\sum_j \ell_j^3 \leq \ell_-^3\,.
$$
By
(\ref{ml0}) and  (\ref{mdallalto}), we end up with
$$\frac{1}{2}\leq \dfrac{0.45(2\pi-\ell_-)}{8\pi(4-\pi)} \ell_-^3.$$
This polynomial inequality provides finally the inequality
$
\ell_-\geq 1.74
$.
\item
Assume that the midpoint of one interval, say $I_1$,  belongs to 
$\displaystyle \left[-\frac{\pi}{2},\frac{\pi}{2}\right]$.  Notice that by Lemma \ref{just_one}, such a negative interval is unique.
We are going to estimate
the sum of 
$
\int_{I_j} (-u(x)) dx
$
on all negative
intervals $I_j$. 
For  $j \neq 1$,  the midpoint of $I_j$ does not belong to $\displaystyle \left[-\frac{\pi}{2},\frac{\pi}{2}\right]$. Therefore
the integral $
\int_{I_j} (-u(x)) dx, j\neq 1
$ 
can be estimated  as in the previous case, that is, using (\ref{gr1n}) and (\ref{stima_tg}).
Instead, $
\int_{I_1} (-u(x)) dx
$ 
can be estimated by 
\begin{equation}\label{termP}
(m-\lambda_0)[2\tan(\ell_1/2) -\ell_1]
+P, \quad P=-\frac{\lambda_1}{2} 2\sin\left(\frac{\ell_1}{2}\right)\cos m_1 \left(1+\frac{\ell_1}{\sin\ell_1}\right)
\end{equation}
by formula (\ref{gr1n}). 
We thus obtain, for the sum
of 
$
\int_{I_j} (-u(x)) dx
$
on all negative
intervals $I_j$
$$\frac{1}{2}=\sum_j \int_{I_j} (-u(x)) dx \leq \frac{0.45(m-\lambda_0)}{4}
\sum_j \ell_j^3 +P\,,
$$
using  (\ref{int12}) for the first equality.
By the same argument as in the previous case, we get
\begin{equation}\label{1/2P}
\frac{1}{2}\leq \dfrac{0.45(2\pi-\ell_-)}{8\pi(4-\pi)} \ell_-^3 + P\,.
\end{equation}
We are going to distinguish two cases, according to the values of $\ell_1$:
\begin{enumerate}
\item
Assume $\ell_1\geq 0.228$.
 In this case, we look at the first negative interval $I_{-1}$ 
on the left of $I_0$. Since $\displaystyle b\leq \frac{\pi-\ell_1}2\leq \frac{\pi}{2}-0.114$
and $\ell\geq \pi$, we have
$$m_{-1}=a-\frac{\ell_{-1}}{2}=b-\ell-\frac{\ell_{-1}}{2} \leq -0.114-\frac{\pi+\ell_{-1}}{2}.$$
If $\displaystyle \ell_{-1}\geq \frac{\pi}{2}$, then $\displaystyle \ell_-\geq
\frac{\pi}{2}+0.228>1.75$. If, on the contrary, 
$\displaystyle \ell_{-1}< \frac{\pi}{2}$, then  $\displaystyle m_{-1}\geq
-\frac{5\pi}{4}$ necessarily (otherwise 
$\displaystyle a=m_{-1}+\frac{\ell_{-1}}2 < -\pi$ which is impossible). Therefore,
$$|\cos(m_{-1})|\geq \min\left\{\frac{\sqrt{2}}{2},\left|\cos\left(0.114+\frac{\pi+\ell_{-1}}{2}\right)\right|\right\}
:=C.$$
Using this estimate and $\ell_-\geq 0.228+\ell_{-1}$
in \eqref{mainn} we get
$$A\leq \dfrac{\ell_{-1}^2(2\pi-0.228-\ell_{-1})}{12(0.228+\ell_{-1})C}.$$
As a function of $\ell_{-1}$, the right-hand side is increasing and for $\ell_{-1}\leq
1.322$
we get the inequality $\displaystyle A\leq 0.636 < \frac2{\pi}$, that is,
a contradiction. Therefore, in this case, we deduce 
$\ell_{-1}\geq 1.322$ and $\ell_-\geq \ell_1+\ell_{-1}\geq 1.55$.
\item 
Assume 
$\ell_1\leq 0.228$. 
Lemma \ref{majol1} and inequalities (\ref{ml0}) and (\ref{mdallalto}) give
$$-\frac{\lambda_1}{2} \leq m+\lambda_0=m\frac{\ell_-}{\pi} \leq \frac{\ell_-}{2\pi(4-\pi)}\,.$$
We can use these inequalities to estimate
the positive term $P$ defined by (\ref{termP}) :
$$P\leq \frac{\ell_-}{2\pi(4-\pi)} 2\sin(0.114) \left(1+\frac{0.228}{\sin(0.228)}\right)\,.$$
Therefore we get from (\ref{1/2P})
$$\frac{1}{2}\leq \dfrac{0.45(2\pi-\ell_-)}{8\pi(4-\pi)} \ell_-^3+
\frac{\ell_-}{2\pi(4-\pi)} 2\sin(0.114) \left(1+\frac{0.228}{\sin(0.228)}\right).$$
This polynomial inequality implies $\ell_-\geq 1.55$. 
\end{enumerate}
\end{enumerate}
\end{proof}
In the next two propositions we will analyse the case where there would be at least 6 nodal domains. 
\begin{proposition}\label{6nodal_domains_first_case}
Assume that $u$ has at least 6 nodal domains and $\displaystyle b\geq \frac{\pi}{2}$.
Then $A<2/\pi$ and the length of no nodal interval is greater than $\pi$.
\end{proposition}
\begin{proof}
 We will consider the two negative
intervals next to $I_0=(a,b)$,
namely $I_{-1}$ on the left of $I_0$ and $I_1$ on the right, and prove that
both have lengths $\ell_{-1}$ and $\ell_{1}$ respectively greater than $1.17$.
\begin{itemize}
\item
Assume by contradiction that $\ell_{-1}\leq 1.17$. On  one hand, the midpoint $m_{-1}$ of $I_{-1}$ satisfies 
$$m_{-1}=a-\frac{\ell_{-1}}{2}\leq -\frac{\pi}{2}-\frac{\ell_{-1}}{2}.$$
On the other hand, $\displaystyle m_{-1}\geq -\frac{5\pi}{4}$ (otherwise
$\displaystyle a=m_{-1}+\frac{\ell_{-1}}{2}<-\pi$ which is impossible). Therefore
we have 
$$|\cos(m_{-1})|\geq \min\left\{\frac{1}{\sqrt{2}},\left|\cos\left(\frac{\pi+\ell_{-1}}{2}\right)\right|\right\}=\min\left\{\frac{1}{\sqrt{2}},\sin\left(\frac{\ell_{-1}}{2}\right)\right\}\,.$$
Plugging this estimate in \eqref{mainn} and using Proposition \ref{lemma155}
yields
\begin{equation}\label{estimateAmiddlepoint}
A\leq \frac{\ell_{-1}^2(2\pi-1.55)}{12\cdot 1.55\cdot \min\left\{\frac{1}{\sqrt{2}},\sin(\ell_{-1}/2)\right\}}.
\end{equation}
Now the right-hand side is increasing in $\ell_{-1}$ and its value for $\ell_{-1}=1.17$
is less than $\displaystyle 0.631<\frac2{\pi}$ that is absurd, proving the
claim. 
\item
Assume by contradiction  that $\ell_{1}\leq 1.17$. In this case we are going
to use that $\displaystyle b\geq \frac{\pi}{2}$ to estimate $m_1$:
 $$\frac{\pi}{2}+\frac{\ell_1}{2} \leq b+\frac{\ell_1}{2}=m_1=b+\frac{\ell_1}{2}<\pi
+\frac{1.17}{2}<\frac{5\pi}{4}\,.
 $$
 One can now repeat the above argument to get a contradiction.
\end{itemize}  
By the previous estimates on $\ell_{-1}, \ell_{1}$, 
we get $\ell_-\geq \ell_{-1}+\ell_1\geq 2.34$.
\\
With the same arguments of the two steps above, we can prove that  
$\displaystyle \ell_{-1}\geq \frac{\pi}{2}$ and $\displaystyle \ell_1\geq
\frac{\pi}{2}$. Indeed it is sufficient to replace 1.17 by  $\displaystyle
\frac{\pi}{2}$
and  $1.55$ by $2.34$ in the estimate of $A$. This gives 
 $\displaystyle A\leq 0.49<\frac2{\pi}$, that is a contradiction. 
We deduce  that $\displaystyle \ell_{-1}\geq \frac{\pi}{2}$ and $\displaystyle
\ell_1\geq \frac{\pi}{2}$ that implies 
$\ell_-\geq \pi$:
this is a contradiction, since $\ell_+>\pi$.
\end{proof}
\begin{proposition}\label{6nodal_domains_second_case}
Assume that $u$ has at least 6 nodal domains and $\displaystyle b< \frac{\pi}{2}$.
Then $A<2/\pi$ and  the length of no nodal interval is greater than $\pi$.
\end{proposition}
The proof of this proposition follows  analogous arguments to the previous one, but uses
 three intervals on the right of $(a,b)$,
since we do not know the position of $m_1$. 
\begin{proof}
One can prove that $\ell_{-1}\geq 1.17$, following exactly the same argument
 as in the proof of Proposition \ref{6nodal_domains_first_case}. 
We will work with the intervals
$I_1=[a_1,b_1], I_2=[a_2,b_2], I_3=[a_3,b_3]$, with $b_1=a_2, b_2=a_3$, on the right of $(a,b)$, with $b=a_1$.
\begin{enumerate}
\item
Assume that $\displaystyle b_2\leq \frac{\pi}{2}$.  Observe that  
$\displaystyle 0<m_2\leq \frac{\pi}{2}-\frac{\ell_2}{2}$
where we have used that  $b>0$. 
Using \eqref{mainp} we get
$$A\leq \frac{\ell_2^2}{12\sin(\ell_2/2)}\,.$$
Since $\displaystyle b_2\leq \frac{\pi}{2}$
and $\displaystyle 0\leq b < a_2=b_2-\ell_2\leq \frac{\pi}{2}-\ell_2$, one
has  $\displaystyle \ell_2\leq \frac{\pi}{2}$.
Therefore
the right hand side of the above estimate of $A$ is less than $\displaystyle
\frac2{\pi}$, that is a contradiction.
\item 
Assume that $\displaystyle b_2=a_3\geq \frac{\pi}{2}$. We start by
proving 
that $\ell_-\geq 1.93$. 
Assume by contradiction that $\ell_-\leq 1.93$.
On one hand we have $\displaystyle m_3=a_3+\frac{\ell_3}2\geq \frac{\pi}{2}+\frac{\ell_3}{2}$.
On the other hand we claim that  
$\displaystyle m_3 <\frac{5\pi}{4}$.
We have $\displaystyle m_3=b+\ell_1+\ell_2+\frac{\ell_3}2$. Observe that
\begin{itemize}
\item $\displaystyle b<\frac{\pi}{2}$ by hypothesis,
\item 
$\displaystyle \ell_2\leq \pi-1.55$
since $\displaystyle \ell_2+\ell\leq \ell_+=2\pi-\ell_-\leq 2\pi-1.55$,
\item 
$\displaystyle \ell_1+\frac{\ell_3}2\leq 0.76$ since
$\displaystyle \ell_{-1}+\ell_1+\frac{\ell_3}2\leq \ell_-\leq 1.93$ and $\ell_{-1}\geq
1.17$.
\end{itemize} 
Thus $\displaystyle m_3 < \frac{3\pi}{2} -1.55 +0.76 <\frac{5\pi}{4}$. We
deduce from the previous bounds on $m_3$ that:
$$|\cos m_3|\geq \min\left\{\frac{1}{\sqrt{2}},\left|\cos\left(\frac{\pi+\ell_3}{2}\right)\right|\right\}$$
which provides, as in  case (a)(i),  $\ell_3\geq 1.17$.
\\
We deduce  that $\ell_-\geq \ell_{-1}+\ell_{3} >1.93$.
\\
We are going to prove that $\ell_{-1}\geq 1.532$. Assume that this is not
the case. Using the same argument as in step (a)(i), one has a contradiction
from the following estimate :
\begin{equation}\label{estimateAcase2}
A\leq \frac{1.532^2(2\pi-1.93)}{12\cdot 1.93\cdot \min\left\{\frac{1}{\sqrt{2}},\sin(1.532/2)\right\}}<\frac{2}{\pi}\,.
\end{equation}
\\
It is easy to prove that  $\ell_{-}>2.67$.  Indeed, if this is not the case,
replacing $1.93$ by $2.67$ in (\ref{estimateAcase2}),  we get a contradiction
since $A$ is still less than $\frac2{\pi}$.
\\
We are going to prove that $\ell_{-1}>\frac{\pi}{2}$. Assume that this is
not the case. By the same argument as in step (a)(i), one has a contradiction
from estimate
$$
A\leq \frac{(\pi/2)^2(2\pi-2.67)}{12\cdot 2.67\cdot \sqrt{2}/2}<\frac{2}{\pi}\,.
$$
\\
We are going to prove that
 $\displaystyle \ell_3\geq \frac{\pi}{2}$. Assume that this is not true.
 Recall that $\ell_{-}<\pi$. 
As above, on one hand we have $\displaystyle m_3=a_3+\frac{\ell_3}{2}\geq
\frac{\pi}{2}+\frac{\ell_3}{2}$.
On the other hand   
$\displaystyle m_3 <\frac{5\pi}{4}$.
Indeed $\displaystyle m_3=b+\ell_1+\ell_2+\frac{\ell_3}{2}$. Observe that
\begin{itemize}
\item $\displaystyle b<\frac{\pi}{2}$ by hypothesis,
\item 
$\displaystyle \ell_2\leq \pi-2.67$
\item 
$\displaystyle \ell_1+\frac{\ell_3}{2}\leq \pi-\frac{\pi}{2}$ 
\end{itemize} 
Thus $\displaystyle m_3 < 2\pi-2.67 <\frac{5\pi}{4}$. We deduce from the
previous bounds on $m_3$ that:
$\displaystyle |\cos m_3|\geq \frac{1}{\sqrt{2}}$.
This gives a contradiction since
$$
A\leq \frac{(\pi/2)^2(2\pi-2.67)}{12\cdot 2.67\cdot \sqrt{2}/2}<\frac{2}{\pi}\,.
$$
 Therefore $\displaystyle \ell_3>\frac{\pi}{2}$. 
\\
This last estimate  implies $\pi\geq \ell_{-}>\ell_{-1}+\ell_3>\pi$, that is, a contradiction.
\end{enumerate}
\end{proof}
%%%%%%%%%%%%%%%%%%%%%%%%%%%%%%%%%%%%%%
We consider now the case of four nodal domains.
\begin{proposition}\label{prop_4nodal_domains}
Assume that $u$ has 4 nodal domains.
Then $A<2/\pi$ and the length of no nodal interval is greater than $\pi$.
\end{proposition}
\begin{proof}
Assume now that $u$ has four nodal domains:
$I_{-1}=I_3, I_1$
(two negative ones)
and  $I_0, I_2$ (two positive ones), that we write as  
$I_0=(a_0,a_1), I_1=(a_1,a_2), I_2=(a_2,a_3), I_{-1}=I_3=(a_3-2\pi,a_0)$.
We assume that $m_0, m_1, m_2, m_3$ are the midpoints of 
each interval and $\ell_0, \ell_1, \ell_2, \ell_3$ are the lengths.

We are going to work with
a more explicit expression of $\lambda_1, \lambda_2$. Without loss of generality,
we can assume 
that the lengths satisfy $0<\ell_1\leq \ell_3<\pi$ (up to replacing $u(x)$
by $u(-x)$).

As in the proof of Proposition \ref{6nodal_domains_first_case}, we can get the lower bound
$\ell_{3}\geq 1.17$.
 We can assume $\displaystyle -\pi \leq m_3 \leq -\frac{\pi}{2}$
{by Lemma \ref{just_one}}.
From $\lambda_2=0$ we have 
\begin{equation}\label{rel0b}
0=\int_{-\pi}^\pi sign(u(x))\sin x dx=2\left(\cos a_0-\cos a_1+\cos a_2-\cos
a_3\right).
\end{equation}
Gathering $-\cos a_1+\cos a_2$ on the one hand and $\cos a_0-\cos a_3$ on
the other hand, the right hand side of \eqref{rel0b}
can be rewritten
\begin{equation}\label{rel1b}
\sin \frac{\ell_3}{2} \sin m_3 + \sin \frac{\ell_1}{2} \sin m_1 =0.
\end{equation}
In the same way, coming back to the definition of $\lambda_1$ (see \eqref{lagrange}),
we have
\begin{equation}\label{rel2b}
\frac{\lambda_1}{2} =\frac{2m}{\pi} \left( \sin \frac{\ell_3}{2} \cos m_3
+ \sin \frac{\ell_1}{2} \cos m_1\right).
\end{equation}
Observing that $\cos m_3<0$, by \eqref{rel1b} $\cos m_3$ can be rewritten
as
\begin{equation}\label{rel3b}
\cos m_3 =-\sqrt{1-\dfrac{\sin^2 \frac{\ell_1}{2} \sin^2 m_1}{\sin^2 \frac{\ell_3}{2}}}\,.
\end{equation}
By using \eqref{rel2b}, \eqref{rel3b} and \eqref{m+l0},
the quantity $A$ defined in \eqref{defA} can be 
rewritten as a function of $\ell_1, \ell_3, m_1$ as
\begin{equation}\label{rel4b}
A=A(\ell_1,\ell_3,m_1)=2\dfrac{\sqrt{\sin^2 \frac{\ell_3}{2} - \sin^2 \frac{\ell_1}{2}
\sin^2 m_1}-\sin \frac{\ell_1}{2} \cos m_1}{\ell_1+\ell_3}\,.
\end{equation}
We use now the integral of $u$ on the interval $I_3$. We recall that
$$ 
\int_{I_3} (-u(x)) dx = \frac{m \ell_+}{\pi}\left(2\tan \frac{\ell_3}{2}
- \ell_3\right)
-\frac{\lambda_1}{2}\left[2\sin \frac{\ell_3}{2} \cos m_3 \left(1+\frac{\ell_3}{\sin
\ell_3}\right)
\right]\,.
$$
Moreover
$\displaystyle 0 < \int_{I_3} (-u(x)) dx \leq \frac{1}{2}$ by \eqref{int12}.
We now use  the expression
of $\lambda_1/2$ given in \eqref{rel2b} and replace $\sin \frac{\ell_1}{2}
\cos m_1$
by $\pm \sqrt{\sin^2 \frac{\ell_1}{2}  - \sin^2 \frac{\ell_3}{2} \sin^2 m_3}$
(with a $+$ if $\displaystyle m_1\leq \frac{\pi}{2}$ and a $-$ if $\displaystyle
\frac{\pi}{2}\leq m_1\leq 1.9$), thanks to (\ref{rel1b}). 
This provides an
expression of the integral as a function of the three variables $\ell_1, \ell_3, m_3$:
$$0<\displaystyle \int_{I_3} (-u(x)) dx= \frac{m}{\pi} I_{\pm}(\ell_1,\ell_3,m_3)\,.$$
More precisely, we have
\begin{eqnarray*}\label{intl1l3m3}
 I_{\pm}(\ell_1,\ell_3,m_3)=(2\pi -\ell_1-\ell_3) \left(2\tan \frac{\ell_3}{2}
- \ell_3\right)
 - \ldots \\
 4\left(\sin \frac{\ell_3}{2} \cos m_3 \pm \sqrt{\sin^2 \frac{\ell_1}{2}
- \sin^2 \frac{\ell_3}{2} 
 \sin^2 m_3}\right) \sin \frac{\ell_3}{2} \cos m_3
 \left(1+\frac{\ell_3}{\sin \ell_3}\right)
\end{eqnarray*}
(with a $+$ if $\displaystyle m_1\leq \frac{\pi}{2}$ and a $-$ if $\displaystyle
\frac{\pi}{2}\leq m_1\leq 1.9$).
By Lemma \ref{studio_funzione_I} below, the function $I_{\pm}$ is negative, and thus  we get a contradiction. 
\end{proof}

We recall that we assume that
$u$ has four nodal domains:
$I_{-1}=I_3, I_1$
(two negative ones)
and  $I_0, I_2$ (two positive ones).
We assume that $m_0, m_1, m_2, m_3$ are the midpoints of 
each interval and $\ell_0, \ell_1, \ell_2, \ell_3$ are the lengths.
We recall that
$$
A=A(\ell_1,\ell_3,m_1)=2\dfrac{\sqrt{\sin^2 \frac{\ell_3}{2} - \sin^2 \frac{\ell_1}{2}
\sin^2 m_1}-\sin \frac{\ell_1}{2} \cos m_1}{\ell_1+\ell_3}\,.
$$
(see \ref{rel4b}) and
\begin{equation}
\cos m_3 =-\sqrt{1-\dfrac{\sin^2 \frac{\ell_1}{2} \sin^2 m_1}{\sin^2 \frac{\ell_3}{2}}}\,.
\end{equation}

\begin{lemma}\label{studio_funzione_I}
Assume  $\ell_3\geq 1.17$ and $0<\ell_1\leq \ell_3<\pi$. 
Let 
 $m_3^*$ be the solution to
$\displaystyle |\sin m_3|\sin \frac{\ell_3}{2}=
\sin \frac{\ell_1}{2}$. Assume that $-\pi\leq m_3\leq m_3^*$.
\begin{enumerate}
\item
Let 
\begin{eqnarray*}\label{intl1l3m3b}
 I_+(\ell_1,\ell_3,m_3)=(2\pi -\ell_1-\ell_3) \left(2\tan \frac{\ell_3}{2}
- \ell_3\right)
 - \ldots \\
 4\left(\sin \frac{\ell_3}{2} \cos m_3 + \sqrt{\sin^2 \frac{\ell_1}{2}
- \sin^2 \frac{\ell_3}{2} 
 \sin^2 m_3}\right) \sin \frac{\ell_3}{2} \cos m_3
 \left(1+\frac{\ell_3}{\sin \ell_3}\right).
\end{eqnarray*}
Assume  $\ell_1\leq \frac{\pi}{6}$. 
Then $I_+(\ell_1,\ell_3,m_3)<0$.
\item
Let 
\begin{eqnarray*}\label{intl1l3m3c}
 I_-(\ell_1,\ell_3,m_3)=(2\pi -\ell_1-\ell_3) \left(2\tan \frac{\ell_3}{2}
- \ell_3\right)
 - \ldots \\
 4\left(\sin \frac{\ell_3}{2} \cos m_3 - \sqrt{\sin^2 \frac{\ell_1}{2}
- \sin^2 \frac{\ell_3}{2} 
 \sin^2 m_3}\right) \sin \frac{\ell_3}{2} \cos m_3
 \left(1+\frac{\ell_3}{\sin \ell_3}\right).
\end{eqnarray*}
Assume  $\ell_1\leq 0.62$. 
Then $I_-(\ell_1,\ell_3,m_3)<0$.
\end{enumerate}
\end{lemma}

The proof of Lemma \ref{studio_funzione_I} is based on the following estimates on $\ell_1$.

\begin{lemma}\label{pi/6}
Assume that $u$ has 4 nodal domains.
Assume that
$\displaystyle m_1\leq \frac{\pi}{2}$.
Then  $\displaystyle \ell_1\leq \frac{\pi}{6}$. 
\end{lemma}
\begin{proof}
Let us first look at the dependence of $A$ with respect to $m_1$. The derivative
of $A(\ell_1,\ell_3,m_1)$ with respect to $m_1$
has the same sign as
$$\sin m_1\left(1-\dfrac{\cos m_1 \sin \frac{\ell_1}{2}}{\sqrt{\sin^2 \frac{\ell_3}{2}
- \sin^2 \frac{\ell_1}{2} \sin^2 m_1}}\right)\,,$$
that is,  the same sign as
$$\sin m_1\left(\sin^2 \frac{\ell_3}{2}  - \sin^2 \frac{\ell_1}{2} \right).$$
The above quantity is positive, since $\displaystyle 0\leq m_1\leq \frac{\pi}{2}$
and $\displaystyle \ell_3\geq \ell_1$. Therefore, if $\displaystyle m_1\leq
\frac{\pi}{2}$
\begin{equation}\label{rel5b}
A(\ell_1,\ell_3,m_1)\leq A(\ell_1,\ell_3,\frac{\pi}{2})=
2\dfrac{\sqrt{\sin^2 \frac{\ell_3}{2}  - \sin^2 \frac{\ell_1}{2} }}{\ell_1+\ell_3}.
\end{equation}
Assume by contradiction that $\displaystyle \ell_1> \frac{\pi}{6}$. We are
going to prove that $\displaystyle A< \frac{2}{\pi}$, thus reaching a contradiction.
We set
$$x=\frac{\ell_3 - \ell_1}{2}\quad\mbox{and}\quad y=\frac{\ell_3 + \ell_1}{2}$$
and observe that
 $\displaystyle 0\leq x<\frac{\pi}{2}$, $\displaystyle \frac{1.17}{2}\leq
y < \frac{\pi}{2}$. The quantity  
 $\displaystyle A(\ell_1,\ell_3,\frac{\pi}{2})$
 can be rewritten as a function of $x, y$: 
$$
G(x,y)=\dfrac{\sqrt{\sin x \sin y}}{y}, \quad (x,y)\in \left[0,\frac{\pi}{2}\right]\times
\left[\frac{1.17}{2},\frac{\pi}{2}\right]\,.
$$
The map $x\mapsto G(x,y)$ is increasing while $y\mapsto G(x,y)$ is decreasing
on this set. We remark that
the assumption $\displaystyle \ell_1 > \frac{\pi}{6}$ is equivalent to $\displaystyle
y > x+\frac{\pi}{6}$. In the rectangle $\displaystyle \left[0,\frac{\pi}{2}\right]\times
\left[\frac{1.17}{2},\frac{\pi}{2}\right]$ 
this implies that $\displaystyle x\in \left[0,\frac{\pi}{3}\right]$.
Now, it can be checked that
$$\forall t\in \left[0,\frac{\pi}{3}\right],\quad G(t,t+\frac{\pi}{6})\leq
\frac{2}{\pi}\,.$$
Indeed the function $t\mapsto G(t,t+\frac{\pi}{6})$ is first increasing then
decreasing and satisfies the above inequality
for its maximum that is approximately at $0.6627206$. From the properties
of $G$, we infer that
$\displaystyle G(x,y) < \frac2{\pi}$ for all $(x,y)$ such that $\displaystyle
y > x+\frac{\pi}{6}, y\in [\frac{1.17}{2},\frac{\pi}{2}]$, $\displaystyle
x\in [0,\frac{\pi}{3}]$, that is, 
$\displaystyle \ell_1  > \frac{\pi}{6}$.
We have thus proved that $\displaystyle A < \frac{2}{\pi}$ as soon as $\displaystyle
\ell_1 > \frac{\pi}{6}$, that is a contradiction.  Therefore $\displaystyle
\ell_1$ is less or equal $\frac{\pi}{6}$.
\end{proof}

\begin{lemma}\label{0.62}
Assume that $u$ has 4 nodal domains.
Assume that
$\displaystyle m_1>\frac{\pi}{2}$.
Then $\ell_1\leq 0.62$.
\end{lemma}
\begin{proof} 
Let us suppose first
that $m_1\geq 1.9$, therefore
$|\cos m_1|\geq |\cos 1.9|$ and, following \eqref{mainn} we infer
$$A(\ell_1,\ell_3,m_1) \leq \frac{l_1^2 (2\pi-\ell_1-\ell_3)}{12 (\ell_1+\ell_3)
|\cos 1.9|}.$$
Now, this expression is decreasing in $\ell_3$ and increasing in $\ell_1$,
thus it is always less than its value
for $\displaystyle \ell_1=\ell_3<\frac{\pi}{2}$:
$$A(\ell_1,\ell_3,m_1) \leq A(\ell_1,\ell_1,1.9)=\frac{l_1 (\pi-\ell_1)}{12
|\cos 1.9|}\leq 
\frac{\pi^2}{48 |\cos 1.9|} < \frac{2}{\pi}$$
and the contradiction is obtained in this case. 

We can therefore assume that 
$\displaystyle \frac{\pi}{2}\leq m_1\leq 1.9$.
Expressing $m_3$ in terms of $m_1$, we have $\displaystyle m_3=m_1-\frac{\ell_-}{2}-\ell_0\leq
1.9 -\frac{\ell_-}{2}-\pi$, thus
$|\sin m_3|\leq \sin (1.9- \frac{\ell_-}{2})$ (recall that $\displaystyle
-\pi\leq m_3\leq -\frac{\pi}{2}$). 
By  identity \eqref{rel1b}, we have
$$
|\sin 1.9| \sin \frac{\ell_1}{2} \leq \sin m_1 \sin \frac{\ell_1}{2} =
|\sin m_3| \sin \frac{\ell_3}{2}
\leq \sin \frac{\ell_3}{2} \sin \left(1.9 -\frac{\ell_-}{2}\right).
$$
This implies, by \eqref{rel3b}
$$|\cos m_3|=\sqrt{1-\dfrac{\sin^2 \frac{\ell_1}{2} \sin^2 m_1}{\sin^2 \frac{\ell_3}{2}}}
\geq \sqrt{1-\dfrac{\sin^2(1.9- \frac{\ell_-}{2} )}{\sin^2 1.9}}.$$
Therefore, using (\ref{mainn})
we can estimate $A$ from above by
$$A\leq \dfrac{\ell_3^2(2\pi-\ell_-)}{12 \ell_- \sqrt{1-\dfrac{\sin^2(1.9-
\frac{\ell_-}{2} )}{\sin^2 1.9}}}.$$
As a function of $\ell_-$ the right-hand side is clearly decreasing. Now,
if $\ell_1\geq 0.62$
(and then $\ell_-\geq 1.17+0.62$), we can see that $\displaystyle A<\frac{2}{\pi}$
for any $\ell_3 \in [1.17, \pi-\ell_1]$
giving the desired contradiction. 
\end{proof}

%%%%%%%%%%%%

%%%%%%%%%%%
%%%%%%%%%%%%%%%%
We are now able to prove Lemma \ref{studio_funzione_I}.
\begin{proof}(of Lemma \ref{studio_funzione_I})
\begin{enumerate}
\item
We look first at the dependence with respect to $m_3$. The derivative of
$I_+(\ell_1,\ell_3,m_3)$ with respect to $m_3$ has the sign of
$$ \sin m_3\left(\left(\sin^2 \frac{\ell_1}{2} - \sin^2 \frac{\ell_3}{2}\sin^2
m_3\right)^{1/4} +
\dfrac{\sin \frac{\ell_3}{2} \cos m_3}{\left(\sin^2 \frac{\ell_1}{2} - \sin^2
\frac{\ell_3}{2}\sin^2 m_3\right)^{1/4}}\right)^2.$$
Since $\displaystyle -\pi\leq m_3\leq -\frac{\pi}{2}$,  $\sin m_3$ is negative.
Therefore $m_3\mapsto I(\ell_1,\ell_3,m_3)$ is decreasing, which implies
$$
I_+(\ell_1,\ell_3,m_3)\leq I_+(\ell_1,\ell_3,-\pi).
$$
Now,
$$
I_+(\ell_1,\ell_3,-\pi)=(2\pi -\ell_1-\ell_3) \left(2\tan \frac{\ell_3}{2}
- \ell_3\right)
 + 4 \left( \sin \frac{\ell_1}{2} - \sin \frac{\ell_3}{2}\right) \sin \frac{\ell_3}{2}
 \left(1+\frac{\ell_3}{\sin \ell_3}\right).$$
The derivative with respect to 
$\ell_1$ is $\displaystyle \ell_3 - 2\tan \frac{\ell_3}{2} +2\cos \frac{\ell_1}{2}
\sin \frac{\ell_3}{2}
 \left(1+\frac{\ell_3}{\sin \ell_3}\right)$ that is decreasing in $\ell_1$,
thus greater than its
 value for $\displaystyle \ell_1=\frac{\pi}{6}$:
$$\frac{\partial I_+(\ell_1,\ell_3,-\pi)}{\partial \ell_1} \geq \ell_3 - 2\tan
\frac{\ell_3}{2} +
2\cos \frac{\pi}{12} \sin \frac{\ell_3}{2}
 \left(1+\frac{\ell_3}{\sin \ell_3}\right) >0.$$
This shows that $I_+(\ell_1,\ell_3,-\pi)$ is increasing in $\ell_1$. We deduce 
that
\begin{itemize}
\item for $\displaystyle 1.17\leq \ell_3\leq \frac{5\pi}{6}, I_+(\ell_1,\ell_3,-\pi)
\leq I_+\left(\frac{\pi}{6},\ell_3,-\pi\right)$;
\item for $\displaystyle \ell_3\in \left[\frac{5\pi}{6}, \pi\right), I_+(\ell_1,\ell_3,-\pi)
\leq I_+(\pi -\ell_3,\ell_3,-\pi)$.
\end{itemize}
The functions on the right hand side of the above inequalities are negative.
\\
\item
We look first at the dependence with respect to $m_3$. The derivative of
$I_-(\ell_1,\ell_3,m_3)$ with respect to $m_3$ has the sign of
$$- \sin m_3\left(\left(\sin^2 \frac{\ell_1}{2} - \sin^2 \frac{\ell_3}{2}\sin^2
m_3\right)^{1/4} -
\dfrac{\sin \frac{\ell_3}{2} \cos m_3}{\left(\sin^2 \frac{\ell_1}{2} - \sin^2
\frac{\ell_3}{2}\sin^2 m_3\right)^{1/4}}\right)^2.$$
Since $\displaystyle -\pi\leq m_3\leq -\frac{\pi}{2}$,  $\sin m_3$ is negative.
Therefore $m_3\mapsto I_-(\ell_1,\ell_3,m_3)$ is  increasing
which implies
$$
I_-(\ell_1,\ell_3,m_3)\leq I_-(\ell_1,\ell_3,m_3^*)\,.
$$
Now
$$I_-(\ell_1,\ell_3,m_3^*)=(2\pi -\ell_1-\ell_3) \left(2\tan \frac{\ell_3}{2}
- \ell_3\right)
 - 4 \left( \sin^2 \frac{\ell_3}{2} - \sin^2 \frac{\ell_1}{2}\right)
 \left(1+\frac{\ell_3}{\sin \ell_3}\right).$$
This is a convex function with respect to $\ell_1$. Therefore
$$I_-(\ell_1,\ell_3,m_3^*) \leq \max\left\{I_-(0,\ell_3,m_3^*),I(0.62,\ell_3,m_3^*)\right\}.$$
Now the functions  $\ell_3\mapsto I_-(0,\ell_3,m_3^*)$ and 
$\ell_3\mapsto I(0.62,\ell_3,m_3^*)$ are decreasing and negative. 
\end{enumerate}
\end{proof}

\end{document}